\documentclass[reqno,11pt]{amsart}
\usepackage{amsmath,amssymb,amsthm,graphicx,a4wide,enumerate,url}
\usepackage[small,bf]{caption} 
\usepackage{subcaption}
\usepackage{amsmath}
\usepackage{color}
\usepackage{lipsum}
\usepackage{amsfonts}
\usepackage{graphicx}
\usepackage{epstopdf}
\usepackage{algorithmic}
\usepackage[outercaption]{sidecap}
\usepackage{amsopn}
\usepackage{enumerate}
\usepackage{array}
\usepackage{makecell}
\usepackage{adjustbox}
\usepackage{calc}
\usepackage{accents}
\usepackage{hyperref}
\usepackage{cleveref}
\usepackage{romannum}
\usepackage{float}
\usepackage{multirow}
\usepackage{verbatim}
\DeclareMathOperator{\sech}{sech}
\DeclareMathOperator{\csch}{csch}

\hypersetup{
  colorlinks   = true, 
  urlcolor     = black, 
  linkcolor    = blue, 
  citecolor   = red 
}

\newtheorem{theorem}{Theorem}[section]

\newtheorem{lemma}[theorem]{Lemma}

\renewcommand{\O}{{\mathcal O}}

\AtBeginDocument{\pagenumbering{arabic}}
\begin{document}
\parskip.9ex

\title[Multiple-Relaxation Runge Kutta Methods for Conservative Systems]
{Multiple-Relaxation Runge Kutta Methods for Conservative Dynamical Systems}

\author[A. Biswas]{Abhijit Biswas}
\address[Abhijit Biswas]
{Computer, Electrical, and Mathematical Sciences \& Engineering Division \\ 
King Abdullah University of Science and Technology \\ Thuwal 23955 \\ Saudi Arabia} 
\email{abhijit.biswas@kaust.edu.sa}
\urladdr{https://math.temple.edu/\~{}tug14809}

\author[D. I. Ketcheson]{David I. Ketcheson}\thanks{Computer, Electrical, and Mathematical Sciences \& Engineering Division,
King Abdullah University of Science and Technology, Thuwal 23955, Saudi Arabia.
This work was supported by funding from the King Abdullah University of Science and Technology.}
\email{david.ketcheson@kaust.edu.sa}
\urladdr{https://www.davidketcheson.info}

\subjclass[2000]{65L04, 65L20, 65M06, 65M12, 65M22.}
\keywords{Runge-Kutta methods, multiple-relaxation RK methods, conservative systems, invariants-preserving numerical methods.}

\begin{abstract}
    We generalize the idea of relaxation time stepping methods in order to preserve multiple nonlinear conserved quantities of a dynamical system by projecting along directions defined by multiple time stepping algorithms.
    Similar to the directional projection method of Calvo et. al., we use embedded Runge-Kutta methods to facilitate this in a computationally efficient manner.
    Proof of the accuracy of the modified RK methods and the existence of valid relaxation parameters are given, under some restrictions. 
    Among other examples, we apply this technique to Implicit-Explicit Runge-Kutta time integration for the Korteweg-de Vries equation and investigate the feasibility and effect of conserving multiple invariants for multi-soliton solutions.
\end{abstract}

\maketitle
\renewcommand{\O}{{\mathcal O}}
\newcommand{\dbtilde}[1]{\accentset{\approx}{#1}}
\newcommand{\vardbtilde}[1]{\tilde{\raisebox{0pt}[0.85\height]{$\tilde{#1}$}}}

\newcommand{\Err}{\mathcal{I}(\zeta)}  
\newcommand{\myremarkend}{$\spadesuit$} 
\newcommand{\lbmap}{\left[} 
\newcommand{\rbmap}{\right]} 
\newcommand{\WSOset}{\mathbb{W}_q}
\newcommand{\OCset}{\mathbb{V}_{p}}
\newcommand{\numc}{n_c}
\newcommand{\minPolc}{P_{\rm C}}
\newcommand{\qso}{\tilde{q}}
\newcommand{\RzCoeff}{\beta}
\newcommand{\Sredr}{\tilde{r}}
\newcommand{\BigO}{{\mathcal{O}}}
\newcommand{\Dt}{\Delta t}
\newcommand{\vecpower}[2]{\vec{#1}^{\,#2}}
\newcommand{\kgen}{m}
\newcommand{\basismat}{L}
\newcommand{\Umat}{U}
\newcommand{\Vmat}{V}
\newcommand{\Wmat}{W}
\newcommand{\bsmall}{\hat{\vec{b}}}
\newcommand{\Ksmall}{\hatK}
\newcommand{\Real}{\mathbb{R}}
\newcommand{\bk}{\vec{b}^{\,k}}
\newcommand{\kibitz}[2]{\textcolor{#1}{#2}}
\newcommand{\abhi}[1] {\kibitz{red}{[AB says: #1]}}
\newcommand{\david}[1] {\kibitz{blue}{[DK says: #1]}}

\section{Introduction}
    The development of structure-preserving time integrators, which preserve qualitative properties of initial value problems, has been a major focus of numerical analysis for the last few decades 
    \cite{haier2006geometric}. Consider an ordinary differential equation (ODE) initial-value problem
    \begin{equation}\label{IVP_1}
             \dot{u}(t)  = f\left(t,u(t)\right), \  u(0) = u_{0}, \ u(t)\in \Real^{m} \;,
    \end{equation}
    where $ \ f : \mathcal{D} \subset \Real \times \Real^m \to \Real^m$ is a sufficiently smooth function. We say that the problem \eqref{IVP_1} has $\ell$ conserved quantities defined in terms of a $\mathcal{C}^{1}(\mathcal{\bar{D}})$ function $ \ G: \mathcal{\bar{D}} \subset \Real^{m} \to \Real^{\ell}$ if
    \begin{align*}
        \frac{\textrm{d}}{\textrm{dt}}G(u(t)) = \nabla G(u(t))^{T}f(t,u(t)) = 0 \;.
    \end{align*}
    Time-dependent differential equations with multiple conserved quantities (mass, energy, etc.) appear in many applications. Examples include special classes of ODEs (for example, the Lotka-Volterra system), Hamiltonian systems,
    and many partial differential equations, including the Korteweg-de Vries (KdV) equation, nonlinear Schr\"{o}dinger  equation, etc. Since in general numerical integrators do not conserve the invariants of these systems, in recent years the preservation of invariants has grown in significance as a criterion of a numerical scheme's effectiveness \cite{li1995finite}. Failure to maintain the invariants sometimes leads to non-physical numerical solutions \cite{gear1992invariants} or spurious blow-up of the numerical solutions \cite{arakawa1997computational} when numerically integrating the system. In many cases, conservative schemes are also proved to have better error growth behavior over time than nonconservative methods \cite{de1997accuracy,duran2000numerical}. In such cases, conservative methods may be the only approach that gives acceptable numerical solutions for long-time simulations.  Further studies illustrating the superiority of conservative approaches over generic methods in this regard can be found in \cite{calvo1993development, cano1997error,calvo2011error,ranocha2021rate}.
    
    Runge-Kutta methods are among the most widely used time integration techniques, and it is natural to consider these methods while studying invariant preserving numerical integrators. All Runge-Kutta methods preserve linear invariants.  Only symplectic RK methods preserve quadratic invariants \cite{cooper1987stability}, and such methods must be fully implicit.  It is also well-known \cite{haier2006geometric} that no Runge-Kutta method can preserve all polynomial invariants of degree $n \geq 3$. This restriction motivates the development of new techniques that can preserve general conserved quantities.
    
    Multiple approaches have been introduced to overcome the shortcomings mentioned above in the existing numerical integrators, including for example, discrete gradients and orthogonal projection.  Here we review the ideas leading to relaxation methods. 
    Dekker and Verwer introduced a modification of the classical 4th-order explicit RK method in order to preserve a quadratic invariant \cite{dekker1984stability}.  This modification can be viewed as projection onto the conservative manifold along an oblique direction.  Similar ideas were further developed and generalized by Calvo and coauthors to develop the directional projection technique, allowing the use of other explicit RK methods and the preservation of multiple, not necessarily quadratic, invariants, by using embedded
    explicit Runge-Kutta methods\cite{calvo2006preservation}.
    A major advantage of these oblique projection techniques when compared with orthogonal projection is that oblique projection can be designed to maintain preservation of linear invariants (which are naturally preserved by RK methods) while also preserving nonlinear invariants.
    Recently, a further modification of these ideas, known as  relaxation \cite{ketcheson2019relaxation,ranocha2020relaxation_1} has been proposed in order to impose dissipation or conservation of general nonlinear functionals while retaining the full accuracy of the original method. This approach was successfully applied to Hamiltonian systems \cite{ranocha2020relaxation} and was further generalized in the context of multistep methods \cite{ranocha2020general}. 
    
    So far, relaxation methods have been used to conserve a single invariant of a system. Preserving a single generic invariant requires solving a nonlinear equation for one relaxation parameter at each time step of the method. In this work we generalize the idea of relaxation in order to preserve multiple invariants, similar to what was done for the incremental direction technique in \cite{calvo2006preservation}.  In order to preserve multiple invariants we require multiple linearly-independent search directions, which could be obtained in principle from any set of distinct time discretizations.  Following the natural and efficient approach used in \cite{calvo2006preservation}, we employ sets of embedded RK methods so that no additional RHS evaluations are needed.
    The resulting approach is similar to that of \cite{calvo2006preservation}, but requires one fewer embedded methods for the same number of conserved quantities.
    Herein we sometimes refer to relaxation methods and the directional projection method of \cite{calvo2006preservation} as \emph{oblique projection methods}.

    In Sections \ref{sec:RRK}-\ref{sec:MRRK}, we develop multiple relaxation methods.  In Section \ref{sec:existence} we prove the existence of solutions to the equations that determine the relaxation parameters, and show that the resulting methods retain the original order of accuracy.  The material in these sections builds closely on previous work on both relaxation and directional projection methods.  In Section \ref{sec:numerical} we verify the effectiveness of the methods on a few ODE examples.  In Section \ref{sec:KdV},
    we take the application of oblique projection methods further than what has been done before, by
    applying them to PDE examples, in combination with IMEX time stepping,
    We study the impact of multiple relaxation on the long-time accuracy of multi-soliton solutions
     and investigate the feasibility of applying relaxation in order to recover conservation laws of the PDE that have been lost in semi-discretization.
    These more challenging applications also shed new light on some of the practical difficulties that may arise in the solution of the algebraic equations used to impose conservation.


\section{Runge-Kutta Methods and Relaxation}\label{sec:RRK}
	An $\mathrm{s}$-stage RK method can be represented via its \emph{Butcher tableau}:
    \begin{equation}\label{B_table_RK}
        \renewcommand\arraystretch{1.2}
        \begin{array}
            {c|c}
            \vec{c} & A\\
            \hline
            & \vec{b}^{\,T}
        \end{array}\;, 
    \end{equation}
    where the matrix $A = (a_{ij})\in \Real^{s \times s}$ and vector $\vec{b} \in \Real^{s}$.  We assume 
    \begin{align} \label{c-assumption}
    \vec{c} = A\vec{e},
    \end{align}
    where $\vec{e}$ is the vector of ones in $\Real^{s}$.
    
	For an initial value problem
    \begin{equation}\label{IVP_2}
             \dot{u}(t)  = f\left(t,u(t)\right), \  u(0) = u_{0}; \ u\in \Real^{m}, \ f : \Real\times\Real^m \to \Real^m,
    \end{equation}
    method \eqref{B_table_RK} provides the numerical solution
    \begin{subequations}\label{RK_methods}
        \begin{align}
            g_i & = u^n+\Delta t \sum_{j = 1}^{s} a_{ij} f(t_n+c_j \Delta t, g_j), \ i = 1,2, \ldots,s \;, \label{RK_methods_a}\\
            u(t_n+\Dt) \approx u^{n+1} & = u^{n}+\Delta t \sum_{j = 1}^{s} b_j f(t_n+c_j \Delta t, g_j)\;, \label{RK_methods_b}
        \end{align}
    where $\Delta t$ is a time step size, $u^n$ and $u^{n+1}$ are the numerical approximations to the true solution at time $t_n$ and $t_{n+1} = t_n+\Delta t$, respectively. Assume that the system \eqref{IVP_2} has a scalar conserved quantity $G_1:\Real^{m} \to \Real$, i.e., $G_1$ remains constant along each solution. In general, RK methods do not preserve this qualitative behavior discretely. To remedy this one may use a slight modification (known as relaxation) \cite{ketcheson2019relaxation,ranocha2020relaxation_1,ranocha2020relaxation} in which the update
    \begin{align}\label{RRK_methods_1}
        u(t_n+\gamma\Dt) \approx u^{n+1}_{\gamma} & = u^{n}+\gamma \Delta t \sum_{j = 1}^{s} b_j f(t_n+c_j \Delta t, g_j),
    \end{align}
    \end{subequations}
    is used instead of \eqref{RK_methods_b}, where $\gamma \in \Real$ is a scalar that is chosen by imposing the discrete conservation property 
    \begin{align} \label{eq:conservation}
       G_1(u^{n+1}_{\gamma})  =  G_1(u^{n}) \;.
    \end{align}
    The nonlinear algebraic equation \eqref{eq:conservation} must be solved at each time step.  For a $p$th-order baseline RK method, it can be shown that there exists a solution satisfying $\gamma = 1+\mathcal{O}(\Delta t^{p-1})$ (under some restrictions) \cite{ketcheson2019relaxation,ranocha2020relaxation_1}. Consequently, the relaxation RK method defined by \eqref{RK_methods_a} and \eqref{RRK_methods_1} has order $p$ when the new updated solution is interpreted as an approximation of the true solution at time $t_n+\gamma \Delta t$. 
    \subsection{Embedded Runge-Kutta Sets}
    In the next section we extend the relaxation idea to enforce conservation of multiple invariants.  In order to do so, we require multiple candidate new solutions $u^{n+1}$ that are all of sufficient accuracy.  A convenient and inexpensive way to obtain such solutions utilizes the concept of embedded RK methods.  

    We say a set of RK methods is embedded if all methods in the set share the same coefficient matrix $A$, with different weight vectors $b$.
    Thus the $k$th method in the set has coefficients
    denoted by $(A,\vec{b}^{\,k})$.  
    The accuracy of a solution involving all methods in the set will be governed by the order of the least accurate method, so we define
    \begin{align}
    p_{\min} = \min_k p_k
    \end{align}
    where $p_k$ is the order of method $(A,\bk)$.
    To apply such a set of methods, the stage equations \eqref{RK_methods_a}
    need only be solved once, after which the simple arithmetic updates
     \begin{align}\label{RK_methods_1}
        u^{n+1,k} & = u^{n}+\Delta t \sum_{j = 1}^{s} b_j^k f(t_n+c_j \Delta t, g_j),
    \end{align} 
    can be computed.  No additional evaluations of $f$ or solutions of algebraic equations are required.
    
    Traditional use of embedded methods has focused on pairs, in which one method is designed to have lower order and the difference between the two solutions is used as an estimate of the local error.  Here, in contrast, we will make use of sets consisting of $\ell\ge 2$ methods, and solutions from all methods in the set will be used as part of the numerical solution update.  Ideally, all methods in the set would have the same order of accuracy.  In the numerical examples in this work, we mainly make use of existing methods (for which the method orders differ), and leave the design of RK sets with equal order of accuracy to future work.
\section{Multiple Relaxation Runge-Kutta Methods}\label{sec:MRRK}
	So far the relaxation approach has been applied to conserve one quantity of a differential system. We generalize this methodology to conserve multiple nonlinear invariants. We assume that the system of differential equations has $\ell$ smooth invariant  functions $G_1(u)$, $G_2(u)$,...,$G_{\ell}(u)$ defined in a solution space in $\mathbb{R}^m$ and define
    \begin{align} \label{l-invariants}
        G := (G_1,G_2, \ldots, G_{\ell})^{T}: \mathbb{R}^m \to \mathbb{R}^{{\ell}} \;.
    \end{align}
    The basic idea to conserve $\ell \geq 1$ invariants is to use a set of ${\ell}$ linearly-independent embedded RK methods (each of order at least 2) and find a suitable direction in the plane spanned by directions induced by those ${\ell}$ methods so that the invariants are preserved:
    \begin{align} \label{relaxation-equations}
        G(u^{n+1}_{\vec{\gamma}}) & = G(u^n).
    \end{align}
    Here the updated solution $u^{n+1}_{\vec{\gamma}}$ is computed using given ${\ell}$ linearly-independent embedded RK methods as
    \begin{subequations}\label{modified_Update_rule}
        \begin{align}
            u(t_n + (1+\Gamma )\Delta t) \approx u^{n+1}_{\vec{\gamma}}  & := u^{n+1}+\Delta  t \sum_{i=1}^{{\ell}} \gamma_{i} d_{i}^{n} \;, \label{modified_Update_rule_a} \\
            d_{k}^{n} & :=\sum_{j = 1}^{s} b_{j}^{k} f(t_n+c_j \Delta t, g_j), \ \text{for} \ k = 1,2,...,{\ell}\;, \label{increments}
        \end{align}
    \end{subequations}
    where $\Gamma = \sum_i \gamma_i$, $u^{n+1}:=u^{n+1,1}=u^{n}+ \Delta t d_{1}^{n}$, $\vec{\gamma} = (\gamma_{1},\gamma_{2},\ldots, \gamma_{{\ell}}) \in \Real^{\ell}$, and stage values $g_i$'s are defined in \eqref{RK_methods_a}. The directions are computed according to given embedded RK methods $(A,\vec{b}^1,\vec{b}^2,\cdots,\vec{b}^{\ell})$
	where, by convention, the first method defined by the vector ${\vec{b}^1} = [b^{1}_{1},b^{1}_{2},\ldots, b^{1}_{s}]^T$ is used to compute $u^{n+1}$. At each step, we now require to solve a nonlinear system of $\ell$ equations in $\ell$ unknowns $(\gamma_1, \gamma_2,…,\gamma_{\ell})$. In the case  $\ell = 1$, this reduces to the usual relaxation approach. We refer to these generalized methods as multiple relaxation Runge-Kutta  (MRRK) methods.

   Method \eqref{modified_Update_rule} is closely related to the directional projection method of \cite{calvo2006preservation}, which also uses a set
   of embedded RK methods.  The main difference is that \eqref{modified_Update_rule} requires one less embedded method, but also requires the adjustment of the updated step time.

\section{Existence of Relaxation Parameters and Accuracy of the Methods}\label{sec:existence}
In this section we prove the existence (when the time step is small enough) of a vector $\vec{\gamma}$ that
satisfies the conservation property and ensures the multiple-relaxation method is accurate to the same order
as the RK method on which it is based.  We start with a result on the size of $\gamma_i$; for related existing
results see \cite[Thm.~4.1]{calvo2006preservation} and \cite[Thm.~2.14]{ranocha2020general}.

   \begin{lemma}
        Suppose that the IVP \eqref{IVP_2} has $\ell \geq 1$ smooth conserved quantities \eqref{l-invariants}, and let a set of embedded Runge-Kutta methods $(A,\vec{b}^{\,1},\dots,\vec{b}^{\,\ell})$, be given such that the order of method $(A,\vec{b}^{\,1})$ is $p\ge 2$. Let $u^n = u(t_n) \in \Real^m$, $D_n = [d_{1}^{n}|\ldots|d_{\ell}^{n}]$  with $d_{k}^{n}$ defined in \eqref{increments}, computed with a time step  $\Delta t \ge 0$, and $u^{n+1} = u^n+\Delta t d_{1}^{n}$. If 
        \begin{align}\label{grad_assumption}
            \nabla G(u^{n+1}) \cdot D_n = B(u^n)\Delta t+\O(\Delta t^2)
        \end{align}     
        holds with non-singular $B(u^n)$, then there exists $\Delta t^{*} > 0$ such that for every $\Delta t \in [0,\Delta t^{*}]$ there is a unique vector $\vec{\gamma}$ 
        such that equations \eqref{relaxation-equations}-\eqref{modified_Update_rule} are satisfied and 
        $$\gamma_{i} = \O(\Delta t^{p-1}), \ \text{for } \ i = 1,2,\cdots, \ell.$$
    \end{lemma}
    \begin{proof}
        The proof is similar to the corresponding part of that of \cite[Thm~4.1]{calvo2006preservation}.
        Consider the real function $g : \mathbb{R}\times\mathbb{R}^{\ell} \to \mathbb{R}^{\ell}$ defined by 
        \begin{align}\label{Lemma_eq1}
            g(\Delta t, \vec{\gamma}) :=      
           \frac{G\left(u^{n+1}+\Delta t \sum_{i=1}^{{\ell}} \gamma_{i} d_{i}^{n}\right)- G(u^{n})}{\Delta t^2}, & \ \text{for} \ \Delta t \neq 0 \;.
        \end{align}
        Note that the numerical solution of an RK method of order $p$ satisfies (with the assumption $u^{n} = u(t_n)$)
        \begin{align} \label{accuracy_diff_G}
            G(u^{n+1})- G(u^{n}) & = G\left(u(t_{n+1})+\O(\Delta t^{p+1})\right) - G_i\left(u(t_{n})\right) \nonumber \\
            & = G\left(u(t_{n+1})\right) - G\left(u(t_{n})\right) +\O(\Delta t^{p+1}) \\
            & = \O(\Delta t^{p+1})\;.  \nonumber
        \end{align}
        Using Taylor's theorem with $h_n  := \sum_{i=1}^{{\ell}} \gamma_{i} d_{i}^{n}$, the assumption \eqref{grad_assumption}, and equation \eqref{accuracy_diff_G} in \eqref{Lemma_eq1} we write
        \begin{align*}\label{Lemma1_eq3}
             g(\Delta t, \vec{\gamma}) & = \Delta t^{-2} \left[ G\left(u^{n+1}+ \Delta t h_n\right) - G(u^{n})\right]  \\
             & = \Delta t^{-2} \left[ G(u^{n+1}) + \Delta t \nabla G(u^{n+1}) \cdot h_{n} + \Delta t^2 \frac{G''(u^{n+1})}{2}(h_{n},h_{n}) +\O(\Delta t^3) - G(u^{n})\right]  \\
             & = \frac{G(u^{n+1})- G(u^{n})}{\Delta t^2} + \Delta t^{-1} \nabla G(u^{n+1}) \cdot D_{n} \vec{\gamma}+\frac{G''(u^{n+1})}{2} (h_{n},h_{n}) +\O(\Delta t)  \\
             & = \O(\Delta t^{p-1}) + \left(B(u^n)+\O(\Delta t)\right)\vec{\gamma}+\frac{G''(u^{n+1})}{2} (h_{n},h_{n})+\O(\Delta t)\\
             & = \left(B(u^n)+\O(\Delta t)\right)\vec{\gamma}+\frac{G''(u^{n+1})}{2} (h_{n},h_{n})+\O(\Delta t) \ \text{since} \ p \geq 2 \;.
        \end{align*}
        Thus we define
        \begin{align}
            g(0,\vec{\gamma}):=\displaystyle{\lim_{\Delta t \to 0} g(\Delta t, \vec{\gamma})} = B(u^n)\vec{\gamma}+\left. \frac{G''(u^{n+1})}{2!} (h_{n},h_{n})\right|_{\Delta t = 0}\;,
        \end{align}
        so that $g(\Delta t,\vec{\gamma})$ is continuous for all $\Delta t \geq 0$. Notice that $g(0,\vec{0}) = 0$. Furthermore, $g$ is differentiable with respect to $\Vec{\gamma}$ and its Jacobian is given by
        $$J_{g,\vec{\gamma}}(0,\vec{0}) = B(u^n).$$
        Since $B(u^n)$ is non-singular by assumption, the implicit function theorem guarantees the existence of a $\Delta t^{*} > 0$ and a unique function $$\vec{\gamma}(\Delta t) = \left(\gamma_1(\Delta t), \gamma_2(\Delta t),\ldots, \gamma_{\ell}(\Delta t)\right)$$
        such that $\vec{\gamma}(0) = \vec{0}$ and for any $\Delta t \in [0,\Delta t^{*}]$ we have $g(\Delta t,\vec{\gamma}(\Delta t)) = 0$. Hence it follows from \eqref{Lemma_eq1} that the method defined in \eqref{modified_Update_rule} applied with this set of relaxation parameters will satisfy \eqref{relaxation-equations}. 
        
        To know the accuracy of $\vec{\gamma}$ consider the expansion
        \begin{align}\label{Lemma1_eq4}
            g(\Delta t,\vec{\gamma}) = g(\Delta t,\vec{0}) + J_{g,\vec{\gamma}}(\Delta t,\vec{0}) \cdot \vec{\gamma} + \mathcal{O}(||\vec{\gamma}||^2) 
        \end{align}
        with 
        \begin{align}\label{Lemma1_eq5}
            g(\Delta t,\vec{0}) = \frac{G(u^{n+1})- G(u^{n})}{\Delta t^2} = \mathcal{O}(\Delta t^{p-1}), \ \text{by} \ \eqref{accuracy_diff_G} \;,
        \end{align} 
        and 
        \begin{align}\label{Lemma1_eq6}
            J_{g,\vec{\gamma}}(\Delta t,\vec{0}) = J_{g,\vec{\gamma}}(0,\vec{0}) + \mathcal{O}(\Delta t) = B(u^n) +\mathcal{O}(\Delta t) \;. 
        \end{align}
        Using \eqref{Lemma1_eq4}, \eqref{Lemma1_eq5}, and \eqref{Lemma1_eq6} we conclude that each component of the vector $\vec{\gamma}$ is $\mathcal{O}(\Delta t^{p-1})$.   
    \end{proof}
    \begin{theorem}
        Suppose the IVP \eqref{IVP_2} has $\ell \geq 1$ smooth conserved quantities \eqref{l-invariants} and let $(A,\vec{b}^{\,1},\dots,\vec{b}^{\,\ell})$ be the coefficients of a set of $\ell$ embedded Runge-Kutta methods with orders $p, p_2,\ldots p_{\ell}$, where $p \geq 2$ and $p_i \geq 1$  for $i=2,\ldots, \ell$. Consider the generealized relaxation method defined by \eqref{relaxation-equations}-\eqref{modified_Update_rule} and assume $\vec{\gamma} = \O(\Delta t^{p-1})$. Then:
        \begin{enumerate}
            \item If the solution $u^{n+1}_{\vec{\gamma}}$ is interpreted as an approximation to $u(t_n+\Delta t)$, the method has order $p-1$.
            \item The generalized relaxation method interpreting $u^{n+1}_{\vec{\gamma}}$ as an approximation to $u(t^{n+1}_{\vec{\gamma}})$ has order $p$, where $t^{n+1}_{\vec{\gamma}}=t_n+\left(1+\Gamma \right)\Delta t$.
        \end{enumerate}
    \end{theorem}
    \begin{proof}
     Using \eqref{increments} and the assumption $u^{n} = u(t_n)$ we can write 
    \begin{equation}\label{Eq:theorem1_1}
        \begin{aligned}
        \Delta t d_{i}^{n} =\Delta t \sum_{j = 1}^{s} b_{j}^{i} f(t_n+c_j \Delta t, g_j) & = u^{n+1,i} - u^{n} \\
        & = u(t_{n+1})+\O(\Delta t^2) - u(t_{n}) \ \text{since each $p_i \geq 1$} \\
        & = \Delta t \dot{u}(t_{n+1})+\O(\Delta t^2) \ \text{using Taylor's theorem} \;.
        \end{aligned}  
    \end{equation}
    
     Using the accuracy of the method $(A,\vec{b}^{\; 1})$, \eqref{Eq:theorem1_1} and the notation $\Gamma = \sum_i \gamma_i$, we obtain
     \begin{equation}\label{Eq:theorem1_2}
        \begin{aligned}
            u^{n+1}_{\vec{\gamma}}  & = u^{n+1}+\Delta  t \sum_{i=1}^{{\ell}} \gamma_{i} d_{i}^{n}  \\
            & = u(t_{n+1})+\O(\Delta t^{p+1}) + \sum_{i=1}^{{\ell}} \gamma_{i} \left(\Delta t \dot{u}(t_{n+1})+\O(\Delta t^2)\right) \\
            & = u(t_{n+1}) + \Gamma \Delta t \dot{u}(t_{n+1})+ \O(\Delta t^{p+1}) + \Gamma \O(\Delta t^2) \;.
        \end{aligned}
     \end{equation}
      This shows that the MRRK method interpreted as an approximation to $u(t_{n+1})$ has order $(p-1)$ since $\Gamma = \O(\Delta^{p-1})$.

    When we interpret the solution of the MRRK method at $t^{n+1}_{\vec{\gamma}}$, we can write using Taylor expansion
    \begin{equation}\label{Eq:theorem1_3}
        \begin{aligned}
            u\left(t^{n+1}_{\vec{\gamma}}\right) & = u(t_{n+1}+\Gamma \Delta t) = u(t_{n+1}) + \Gamma \Delta t \dot{u}(t_{n+1}) + \O\left(\left(\Gamma \Delta t\right)^2\right) \;.
        \end{aligned}
    \end{equation}
    Subtracting \eqref{Eq:theorem1_3} from \eqref{Eq:theorem1_2} results in $u^{n+1}_{\vec{\gamma}} - u\left(t^{n+1}_{\vec{\gamma}}\right) = \O(\Delta t^{p+1})$ since $\Gamma = \O(\Delta^{p-1})$, and hence the MRRK method has order $p$.  
    \end{proof}

\section{Numerical ODE Examples} \label{sec:numerical}
    This section illustrates the effects of multiple relaxation RK methods on several problems with multiple invariants. The following numerical schemes with embedded methods are used in the numerical experiments.
    
    \begin{itemize}
        \item SSPRK(2,2): Two-stage, second-order SSP method \cite{shu1988efficient} with a first-order embedded method (appendix~\ref{app:butcherTableau}).
        \item SSPRK(3,3): Three-stage, third-order SSP method \cite{shu1988efficient} with second-order embedded methods (appendix~\ref{app:butcherTableau}).
        \item Heun(3,3): Three-stage, third-order Heun's method \cite{heun1900neue} with a second-order embedded method (appendix~\ref{app:butcherTableau}).
        \item RK(4,4): Classical four-stage, fourth-order method with a second-order embedded method (appendix~\ref{app:butcherTableau}).
        \item Fehlberg(6,4): Six-stage, fourth-order Fehlberg's method \cite{fehlberg1969low} with third-order embedded methods (appendix~\ref{app:butcherTableau}).
        \item Fehlberg(6,5): Six-stage, fifth-order Fehlberg's method with a fourth-order embedded method \cite{fehlberg1969low}.
        \item DP(7,5): Seven-stage, fifth-order method \cite{prince1981high} with fourth-order embedded methods (appendix~\ref{app:butcherTableau}).
        \item ARK3(2)4L[2]SA: Four-stage, third-order additive Runge-Kutta (ARK) method with a second-order embedded ARK method \cite[Appendix C]{kennedy2003additive}. 
        \item ARK4(3)6L[2]SA: Six-stage, fourth-order ARK method with a third-order embedded ARK method \cite[Appendix C]{kennedy2003additive}.
    \end{itemize}
    Given a set of embedded RK methods, the MRRK method defined in \eqref{modified_Update_rule} requires solving a small system of nonlinear equations for the relaxation parameters at each time step. We use the general nonlinear solver \verb|scipy.optimize.fsolve | to solve the nonlinear system for most of the numerical examples below. In some cases, a combination of optimizers \verb|scipy.optimize.brute| and \verb|scipy.optimize.fmin| is used to find the best solution for the relaxation parameters. The update rule \eqref{modified_Update_rule} for the MRRK method, unlike in \eqref{RRK_methods_1}, uses $u^{n+1}$ instead of $u^{n}$, which is observed to provide better robustness in finding the relaxation parameters. Note that, in different cases, sometimes different time steps are required for some methods (see numerical examples) to guarantee the existence of relaxation parameters at all time steps in the simulation. Implementations for all the numerical examples below are provided in \cite{MRRK_code}. We also applied the directional projection technique from \cite{calvo2006preservation} and found comparable results to MRRK methods. However, MRRK methods have the advantage of needing one fewer embedded method.
    \subsection{Rigid Body Rotation}
    Consider the system of Euler equations that describes the motion of a free rigid body with its center of mass at the origin, in terms of its angular momenta:
    \begin{subequations}\label{Eq:Euler_equations}
        \begin{align}
            &\dot{u}_{1}=(\alpha - \beta)u_2u_3 \\
            &\dot{u}_{2}=(1-\alpha)u_3u_1  \\
            &\dot{u}_{3}=(\beta-1)u_1u_2 \;,
        \end{align}
    \end{subequations}
    with $\left(u_1(0),u_2(0),u_3(0)\right)^T = (0,1,1)^T$, $\alpha = 1 + \frac{1}{\sqrt{1.51}}$, and $\beta = 1 - \frac{0.51}{\sqrt{1.51}}$. The exact solution is
    \begin{align}
        \left(u_1(t),u_2(t),u_3(t)\right)^T = \left(\sqrt{1.51} \  \textrm{sn}(t,0.51),\textrm{cn}(t,0.51),\textrm{dn}(t,0.51)\right)^T \;,
    \end{align}
    where $\textrm{sn},\ \textrm{cn},\ \text{and} \ \textrm{dn}$ are the elliptic Jacobi functions. This problem has two quadratic conserved quantities:
    \begin{subequations}\label{Eq:Euler_invariantas}
       \begin{align}
            & G_1(u_{1},u_{2},u_{3}) = u_{1}^{2}+ u_{2}^{2}+u_{3}^{2} \;, \\
            & G_2(u_{1},u_{2},u_{3}) = u_{1}^{2}+ \beta u_{2}^{2}+\alpha u_{3}^{2} \;.
        \end{align}
    \end{subequations}
    Here $G_{2}$ is the kinetic energy of the body. First, we present the convergence results obtained with different RK and MRRK methods to confirm that relaxation with multiple invariants also produces the desired order of convergence. Using four different RK schemes (SSPRK(2,2), Heun(3,3), RK(4,4), and DP(7,5)) and their multiple relaxation versions, we solve the system to a final time of $5$ and report the convergence results in Figure ~\ref{fig:EE_Error_Convergence}. Note that we get better order of accuracy than the theoretically expected orders for all the MRRK methods.
    \begin{figure}
     \centering
     \begin{subfigure}[b]{0.48\textwidth}
         \centering
         \includegraphics[width=\textwidth]{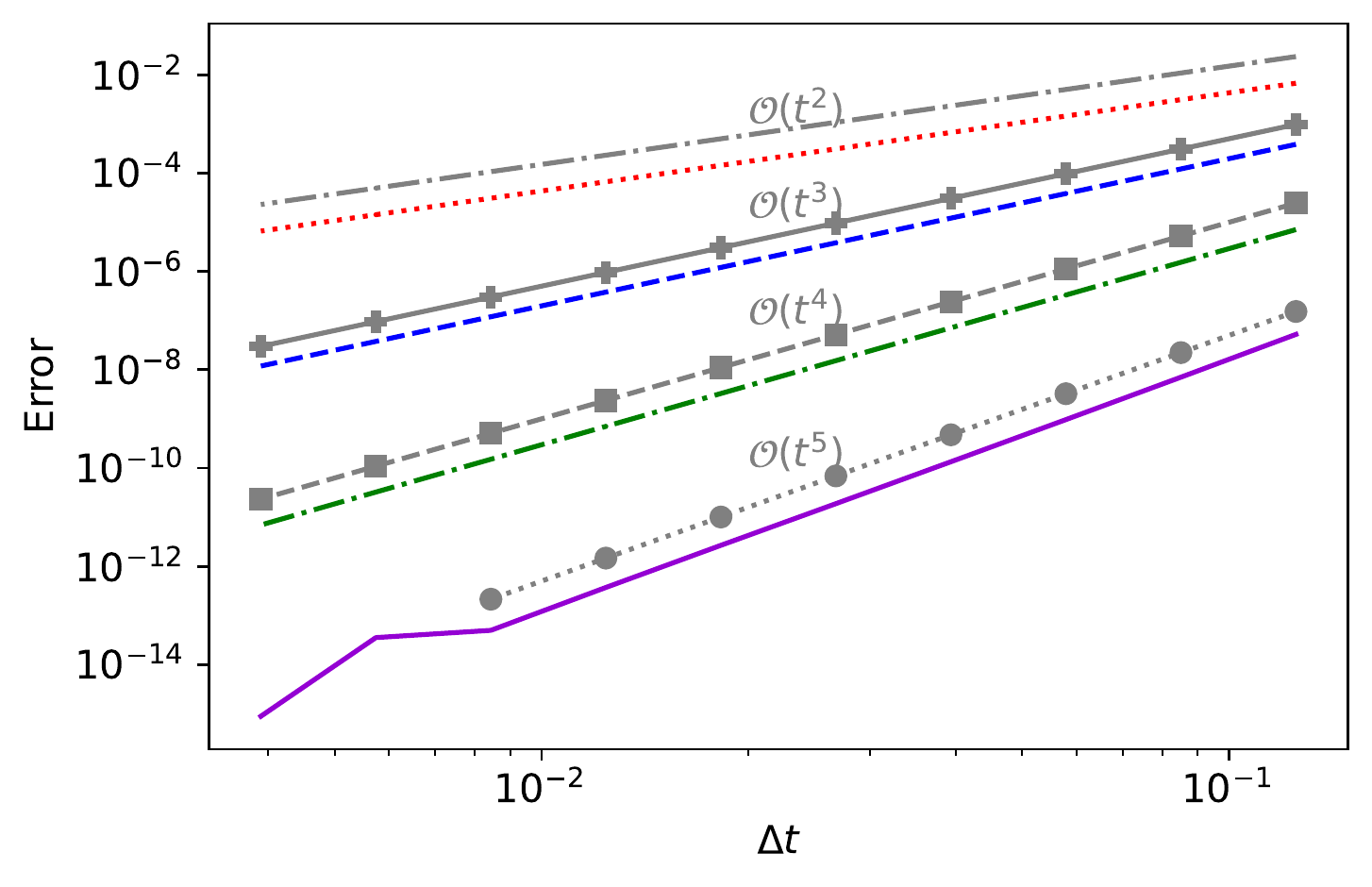}
         \caption{Baseline methods.}
     \end{subfigure}
     \hfill
     \begin{subfigure}[b]{0.48\textwidth}
         \centering
         \includegraphics[width=\textwidth]{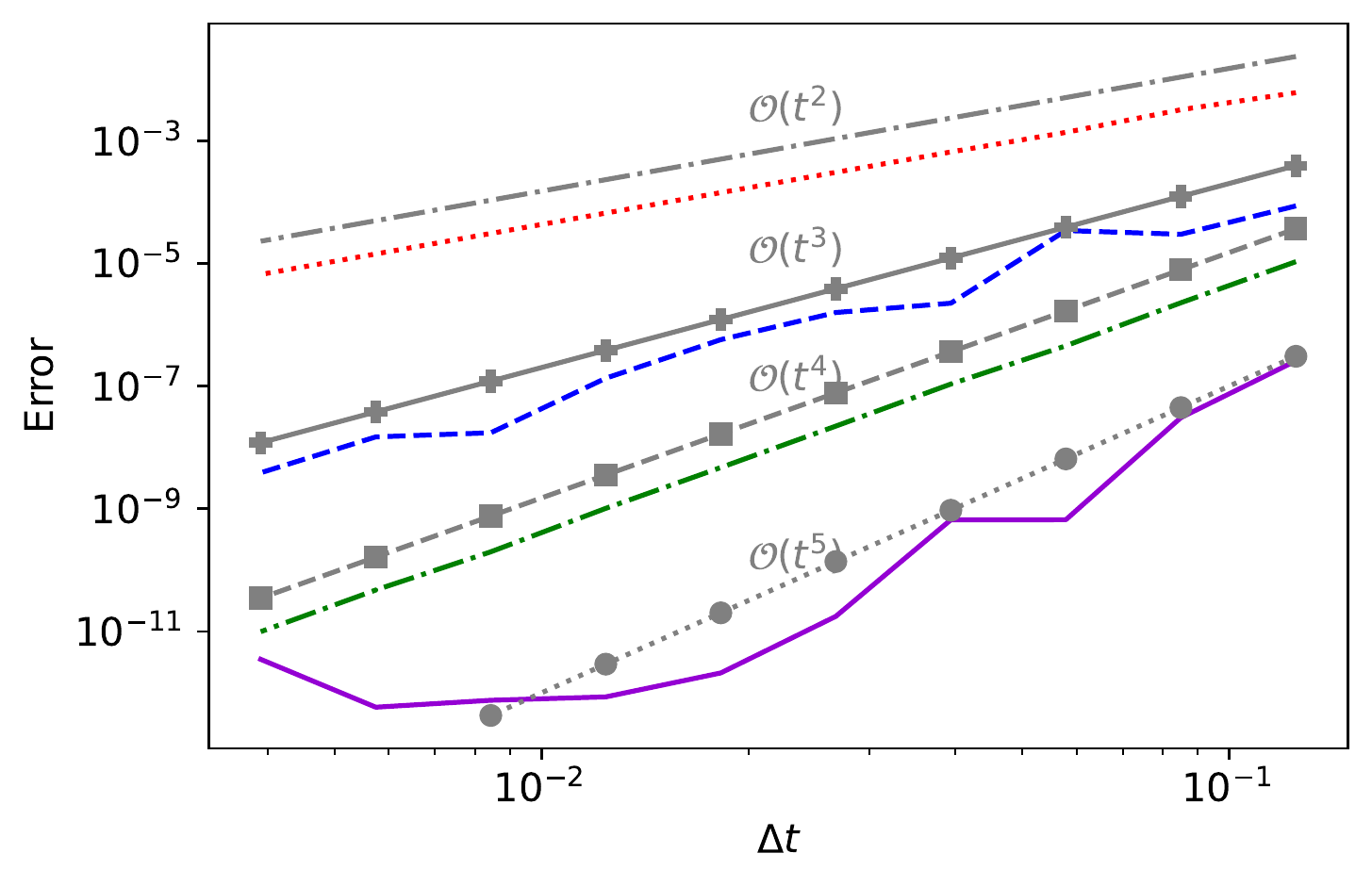}
         \caption{Multiple relaxation RK methods.}
     \end{subfigure}
    	\caption{Convergence of numerical solution by different RK methods and their multiple relaxation versions for rigid body rotation \eqref{Eq:Euler_equations}.}
    	\label{fig:EE_Error_Convergence}
    \end{figure}

    \begin{figure}
        \centering
        \includegraphics[width=\textwidth]{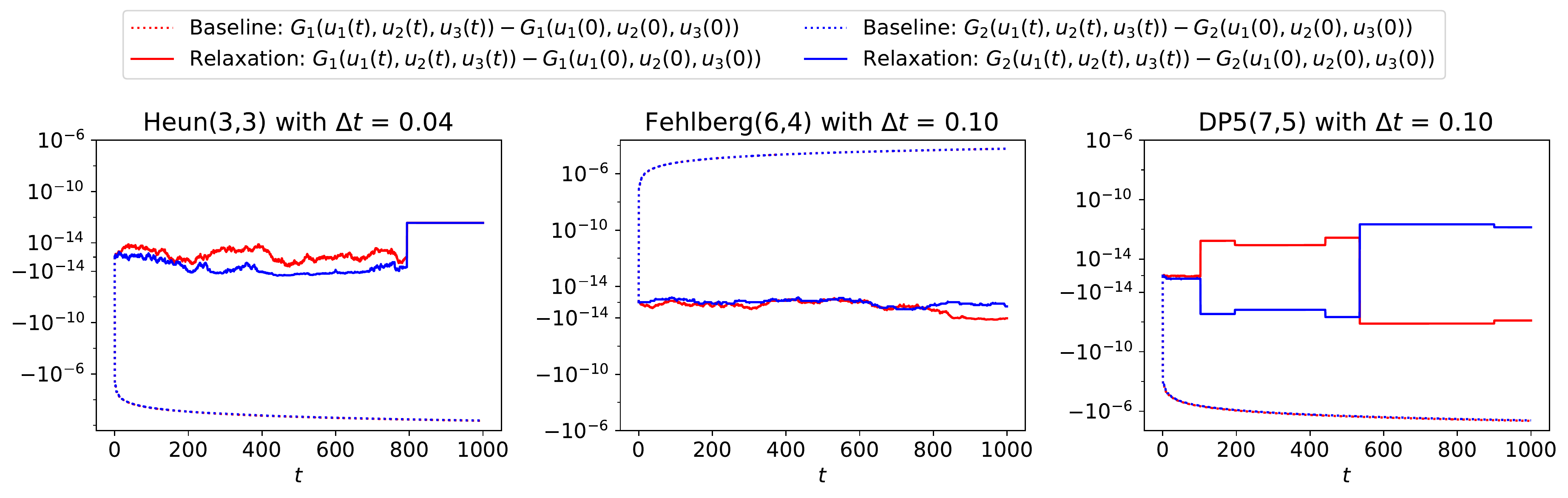}  
        \caption{Changes in invariant \eqref{Eq:Euler_invariantas} obtained with different methods for rigid body rotation \eqref{Eq:Euler_equations}.}
        \label{fig:EE_Evolution_conserved_quantities}
    \end{figure}
    
    \begin{figure}
        \centering
        \includegraphics[width=\textwidth]{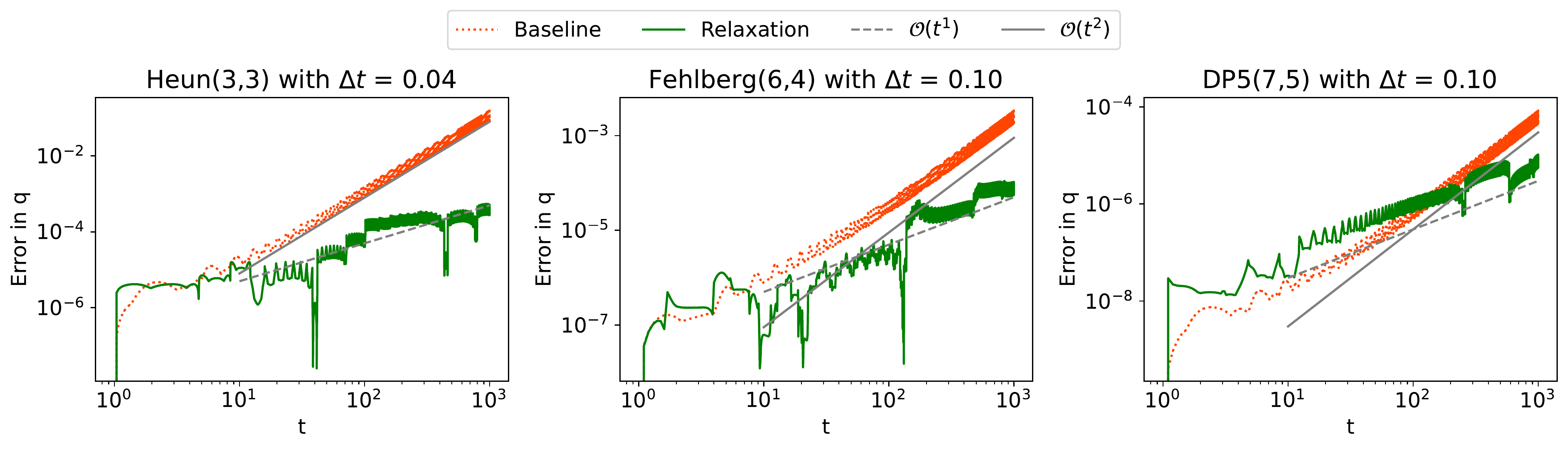} 
    	\caption{Error growth over time for rigid body rotation \eqref{Eq:Euler_equations}.}
    	\label{fig:EE_error_growth}
    \end{figure}
    Next, we study the error in invariants and the error growth over time by various methods. We integrate the problem with three explicit RK methods and their multiple relaxation versions, Heun(3,3) with $\Delta t = 0.04$, Fehlberg(6,4) with $\Delta t = 0.1$, and DP(7,5) with $\Delta t = 0.1$. Figure ~\ref{fig:EE_Evolution_conserved_quantities} plots the error in conserved quantities, confirming that all the MRRK methods preserve both invariants. The solution error growth by all the methods is plotted in Figure ~\ref{fig:EE_error_growth}, which shows a linear error growth when the two invariants are preserved by the MRRK methods. In contrast, the baseline methods produce errors that increase quadratically.

    \subsection{Bi-Hamitonian 3D Lotka-Volterra Systems}
    Next, we consider an ecological predator-prey model of the Lotka-Volterra systems in $3$D, whose equations are given as
    \begin{subequations}\label{Eq:3D_LVS}
        \begin{align}
            &\dot{u}_{1}=u_{1}\left(c u_{2}+u_{3}+\lambda\right) \\
            &\dot{u}_{2}=u_{2}\left(u_{1}+a u_{3}+\mu\right) \\
            &\dot{u}_{3}=u_{3}\left(b u_{1}+u_{2}+\nu\right) \;,
        \end{align}
    \end{subequations}
    where $\lambda, \ \mu, \ \nu >0$, $abc=-1$ and $\nu = \mu b - \lambda ab$. We study this problem on the interval $[0,400]$ with parameters $(a,b,c,\lambda,\mu,\nu) = (-1,-1,-1,0,1,-1)$ and the initial condition is taken as $\left(u_1(0),u_2(0),u_3(0)\right)^T = (1,1.9,0.5)^T$. This system has periodic solutions and possesses two nonlinear conserved quantities, known as Casimirs of some skew-symmetric Poisson matrices \cite{ionescu2015geometrical,uzunca2019preservation}:
        \begin{subequations} \label{Eq:3D_LVS_invariants}
            \begin{align}
            & H_1 = ab \ln{u_1} - b \ln{u_2} + \ln{u_3} \;, \\
            & H_2 = ab u_1 + u_2 - a u_3 + \nu \ln{u_2} - \mu \ln{u_3}.
        \end{align}
    \end{subequations}
    
    \begin{figure}
     \centering
        \includegraphics[width=\textwidth]{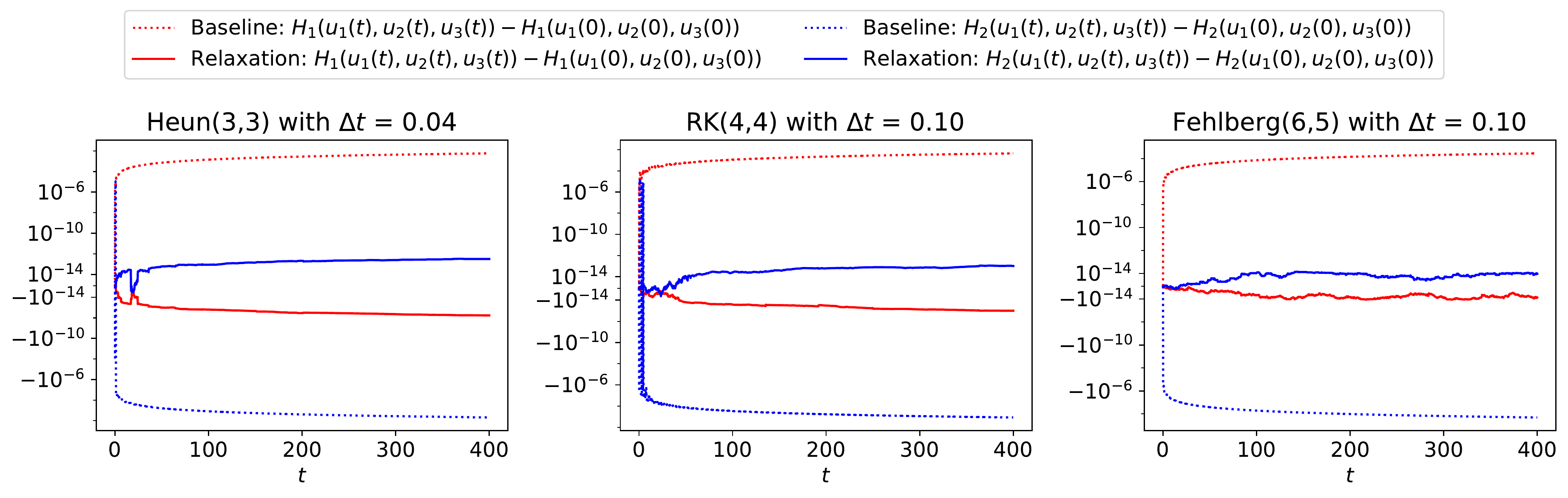}  
    	\caption{Changes in invariant \eqref{Eq:3D_LVS_invariants} obtained with different methods for a bi-Hamiltonian $3$D Lotka-Volterra system \eqref{Eq:3D_LVS}.}
    	\label{fig:3DLVS_Evolution_conserved_quantities}
    \end{figure}
    We apply Heun(3,3) with $\Delta t = 0.04$, RK(4,4) with $\Delta t = 0.1$, and Fehlberg(6,5) with $\Delta t = 0.1$ to solve the system with and without multiple relaxation. All baseline and MRRK methods preserve the periodicity of the solution. The errors in invariants are shown in Figure ~\ref{fig:3DLVS_Evolution_conserved_quantities}, and we can see that all the MRRK methods preserve the nonlinear invariants for the system over a long time. 
    \begin{figure}
     \centering
     \includegraphics[width=\textwidth]{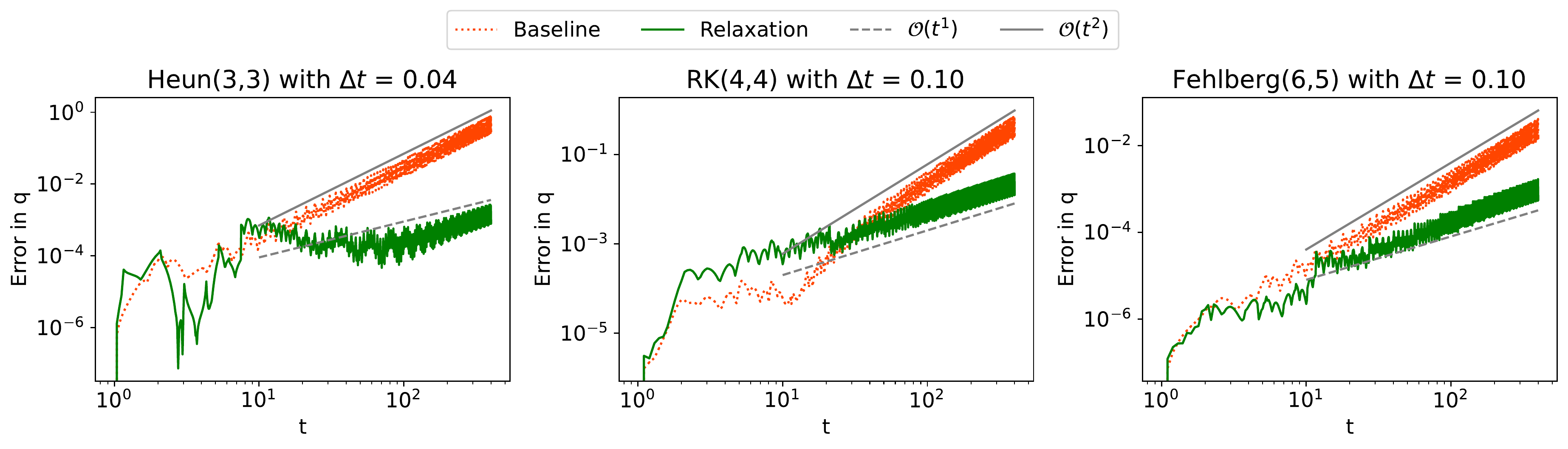}  
     	\caption{Error growth over time for a bi-Hamiltonian $3$D Lotka-Volterra system \eqref{Eq:3D_LVS}.}
    	\label{fig:3DLVS_error_growth}
    \end{figure}
    As the closed form of the analytical solution of this system is not known, as a proxy for the exact solution we use the dense output of the Python interface class \verb|scipy.integrate.ode| with \verb|dopri5| method with minimal values for the relative and absolute tolerances. We measure the error in the maximum norm and plot the error over time in Figure ~\ref{fig:3DLVS_error_growth}. The invariant-preserving MRRK methods show asymptotically linear error growth and eventually win in solution accuracy over the quadratically increasing errors of the corresponding baseline methods. 
    
    \subsection{Kepler's Problem}
    \subsubsection{Kepler's Two-Body Problem with Three Invariants}
    So far, each of the numerical examples considered above involves two conserved quantities. We now study Kepler's Two-Body problem with three invariants to demonstrate that the relaxation process can conserve more than two invariants for a system. With one of the two bodies fixed at the center of the $2$D plane, the motion of the other body with position $q = (q_{1},q_{2})$ and momentum $p = (p_{1},p_{2})$ is given by the following system of first order differential equations
    \begin{subequations}\label{Eq:Kepler_two_body_problem}
        \begin{align}
            \dot{q_{1}} & = p_{1} \\
            \dot{q_{2}} & = p_{2} \\
            \dot{p_{1}} & = -\frac{q_{1}}{\left(q_{1}^2+q_{2}^2\right)^{3/2}} \\
            \dot{p_{2}} & =-\frac{q_{2}}{\left(q_{1}^2+q_{2}^2\right)^{3/2}} \;.
        \end{align}
    \end{subequations}
    Three conserved quantities for Kepler's Two-Body system that we consider for our numerical studies are
    \begin{subequations} \label{Eq:Kepler_invariants}
        \begin{align}
            H(q,p) & = \frac{1}{2}\left(p_{1}^2+p_{2}^2\right)-\frac{1}{\sqrt{q_{1}^2+q_{2}^2}} \ (\text{Hamiltonian}) \;, \\
        	L(q,p) & = q_{1}p_{2}-q_{2}p_{1}\ (\text{angular momentum}) \;, \\
        	A(q,p) & = ||V||_2 \;,
        \end{align}
    \end{subequations}
   where in the last invariant, the well-known Laplace–Runge–Lenz vector function $V$ \cite[Page 26]{haier2006geometric} is defined as 
  \begin{align}
        V & = \begin{bmatrix}
        p_{1} \\
        p_{2} \\
        0
        \end{bmatrix} \times \begin{bmatrix}
        0 \\
        0 \\
        q_{1} p_{2}-q_{2} p_{1}
        \end{bmatrix}-\frac{1}{\sqrt{q_{1}^{2}+q_{2}^{2}}} \begin{bmatrix}
        q_{1} \\
        q_{2} \\
        0
        \end{bmatrix}\;.
   \end{align}
    We consider the Two-Body problem with initial condition $$ \left(q_{1}(0),q_{2}(0),p_{1}(0),p_{2}(0)\right)^{T}=\left(1-e,0,0,\sqrt{\frac{1+e}{1-e}}\right)^{T}$$ with $e=0.5$, and study the conservation of invariants \eqref{Eq:Kepler_invariants} and the global error. Three explicit methods, SSPRK(3,3) with $\Delta t = 0.05$, Fehlberg(6,4) with $\Delta t = 0.05$, and DP(7,5) with $\Delta t = 0.1$ are used as baseline methods to solve the problem.
   \begin{figure}
        \centering
        \includegraphics[width=\textwidth]{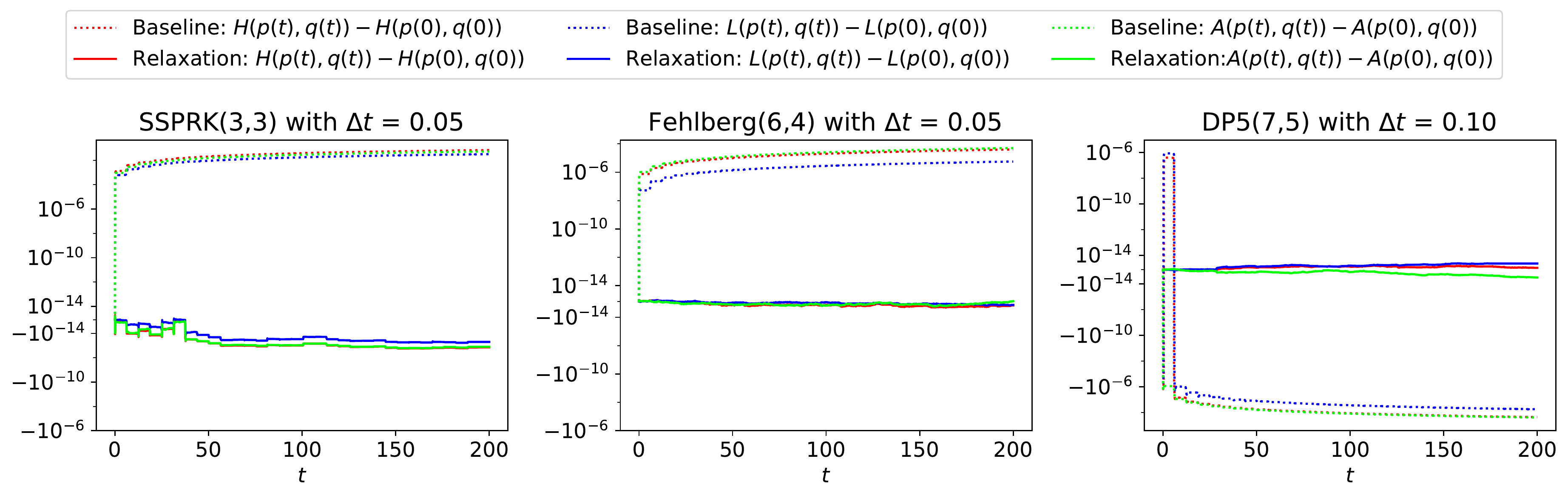}
        \caption{Change in conserved quantities \eqref{Eq:Kepler_invariants} obtained with different methods for Kepler's Two-Body problem \eqref{Eq:Kepler_two_body_problem}.}
        \label{fig:Kepler_Evolution_conserved_quantities}
    \end{figure}
     
    \begin{figure}
     \centering
     \includegraphics[width=\textwidth]{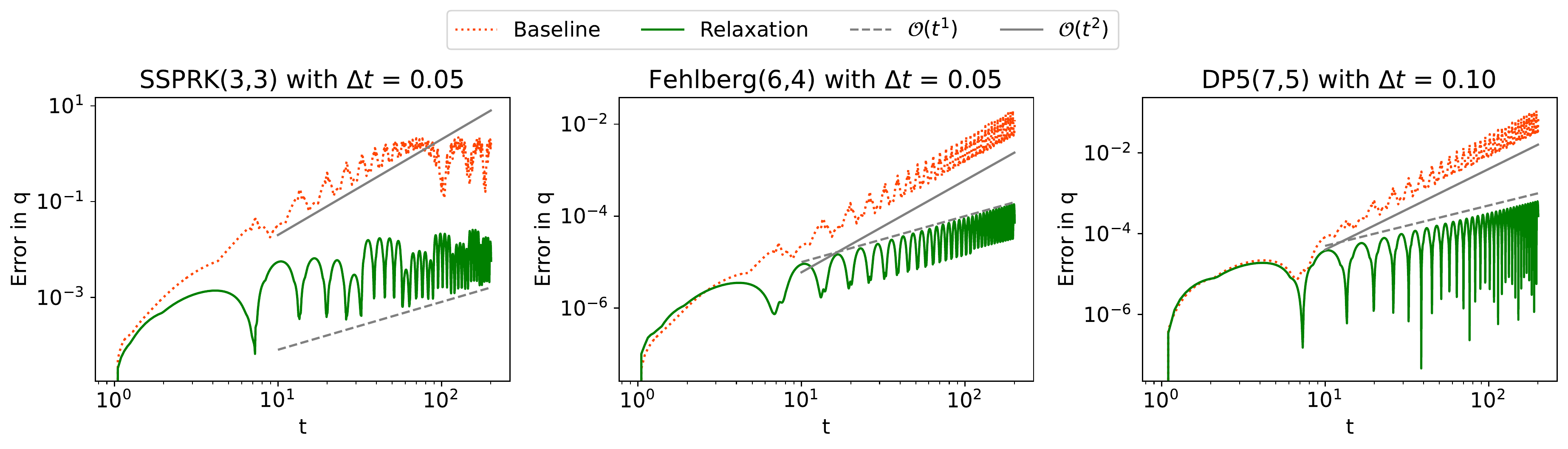}
    	\caption{Error growth over time for Kepler's Two-Body problem \eqref{Eq:Kepler_two_body_problem}.}
    	\label{fig:Kepler_error_growth}
    \end{figure}
    Together with the baseline methods, the relaxation versions of these methods now require two more "independent" embedded methods to solve the system as it has three conserved quantities. The new embedded methods are provided in the appendix ~\ref{app:butcherTableau}, each having order of accuracy one less than that of the corresponding baseline method. Figure ~\ref{fig:Kepler_Evolution_conserved_quantities} demonstrates the advantage of relaxation over baseline RK methods. All the MRRK methods conserve three quantities almost to machine precision, while the baseline RK methods conserve them to around four decimal places. The consequence of these results is reflected in the asymptotic error growth by these methods, shown in Figure ~\ref{fig:Kepler_error_growth}. It shows a quadratic error growth by the baseline methods, while the corresponding MRRK methods achieve a linear error growth over a long time.
    
    \subsubsection{Perturbed Kepler's Problem}
    The governing equations of the perturbed Kepler's problem \cite{haier2006geometric} are given by the following system
    \begin{subequations}\label{Eq:Perturbed_Kepler_problem}
        \begin{align}
            \dot{q_{1}} & = p_{1} \\
            \dot{q_{2}} & = p_{2} \\
            \dot{p_{1}} & = -\frac{q_{1}}{\left(q_{1}^2+q_{2}^2\right)^{3/2}} -\mu \frac{q_{1}}{\left(q_{1}^2+q_{2}^2\right)^{5/2}} \\
            \dot{p_{2}} & =-\frac{q_{2}}{\left(q_{1}^2+q_{2}^2\right)^{3/2}}  -\mu\frac{q_{2}}{\left(q_{1}^2+q_{2}^2\right)^{5/2}} \;,
        \end{align}
    \end{subequations}
    where $\mu$ is a small number taken as $0.005$ for our numerical studies. Previous studies of this problem show that it is important for a numerical method to preserve both the invariants
    \begin{subequations} \label{Eq:Perturbed_Kepler_invariants}
        \begin{align}
            H(q,p) & = \frac{1}{2}\left(p_{1}^2+p_{2}^2\right)-\frac{1}{\sqrt{q_{1}^2+q_{2}^2}}-\frac{\mu}{2\sqrt{\left(q_{1}^2+q_{2}^2\right)^3}}\ (\text{Hamiltonian}) \;, \\
        	L(q,p) & = q_{1}p_{2}-q_{2}p_{1}\ (\text{angular momentum}) \;,
        \end{align}
    \end{subequations}
    to capture the correct behavior of the solution. With the same initial conditions as in Kepler's two-body problem above but with a different eccentricity $e=0.6$, we solve the system using the baseline and the relaxation versions of the methods SSPRK(3,3) with $\Delta t = 0.05$, Fehlberg(6,4) with $\Delta t = 0.05$, and DP(7,5) with $\Delta t = 0.1$. The analytical solution is not available, so we instead use the dense output of Python interface class \verb|scipy.integrate.ode| with \verb|dopri5| method with the relative and absolute tolerances both equal to $10^{-16}$. The errors in both invariants and the numerical solutions of the orbits are presented in Figure ~\ref{fig:Perturbed_Kepler_Evolution_conserved_quantities} and Figure ~\ref{fig:Per_Kepler_error_growth}, respectively.

  \begin{figure}
    \centering
    \includegraphics[width=\textwidth]{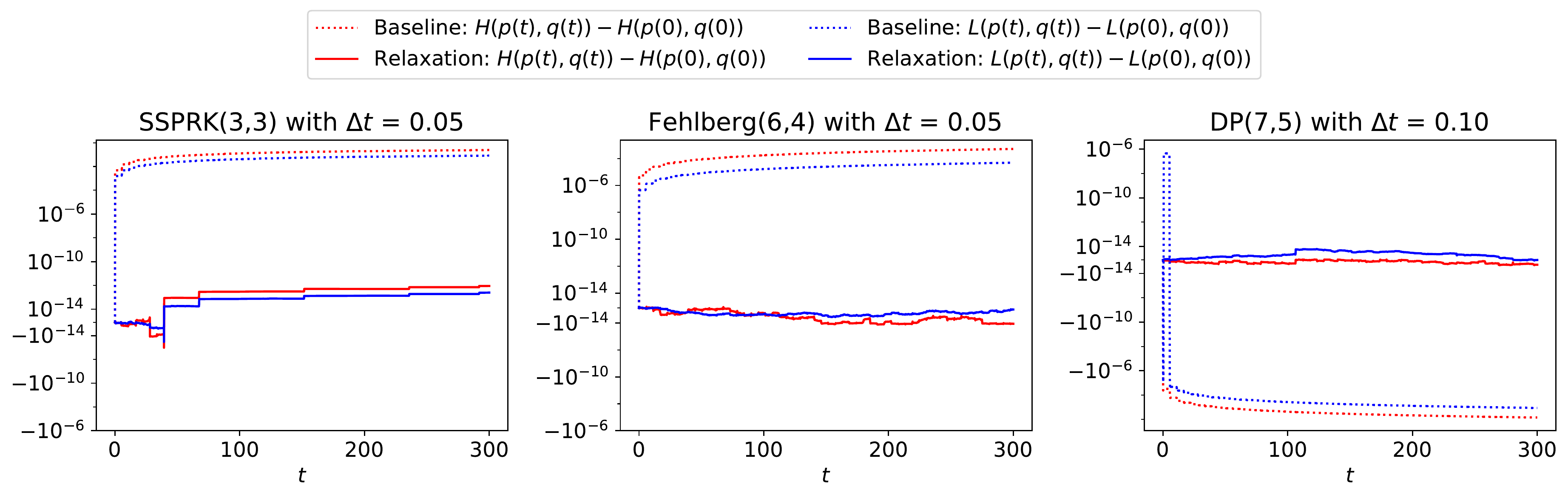}
    \caption{Changes in conserved quantities \eqref{Eq:Perturbed_Kepler_invariants} by different methods for perturbed Kepler's problem \eqref{Eq:Perturbed_Kepler_problem}.}
    \label{fig:Perturbed_Kepler_Evolution_conserved_quantities}
    \end{figure}
    
    \begin{figure}
    \centering
    \begin{subfigure}[b]{0.32\textwidth}
        \centering
        \includegraphics[width=\textwidth]{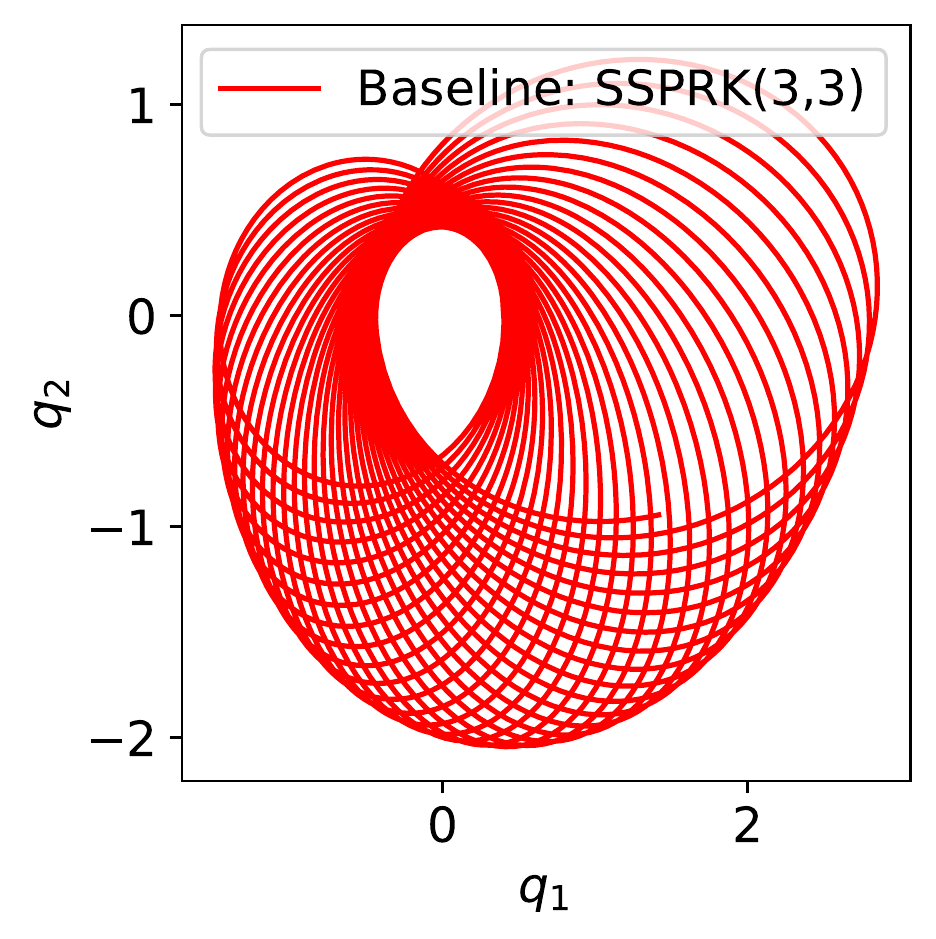}
    \end{subfigure}
    \hfill
    \begin{subfigure}[b]{0.32\textwidth}
        \centering
        \includegraphics[width=\textwidth]{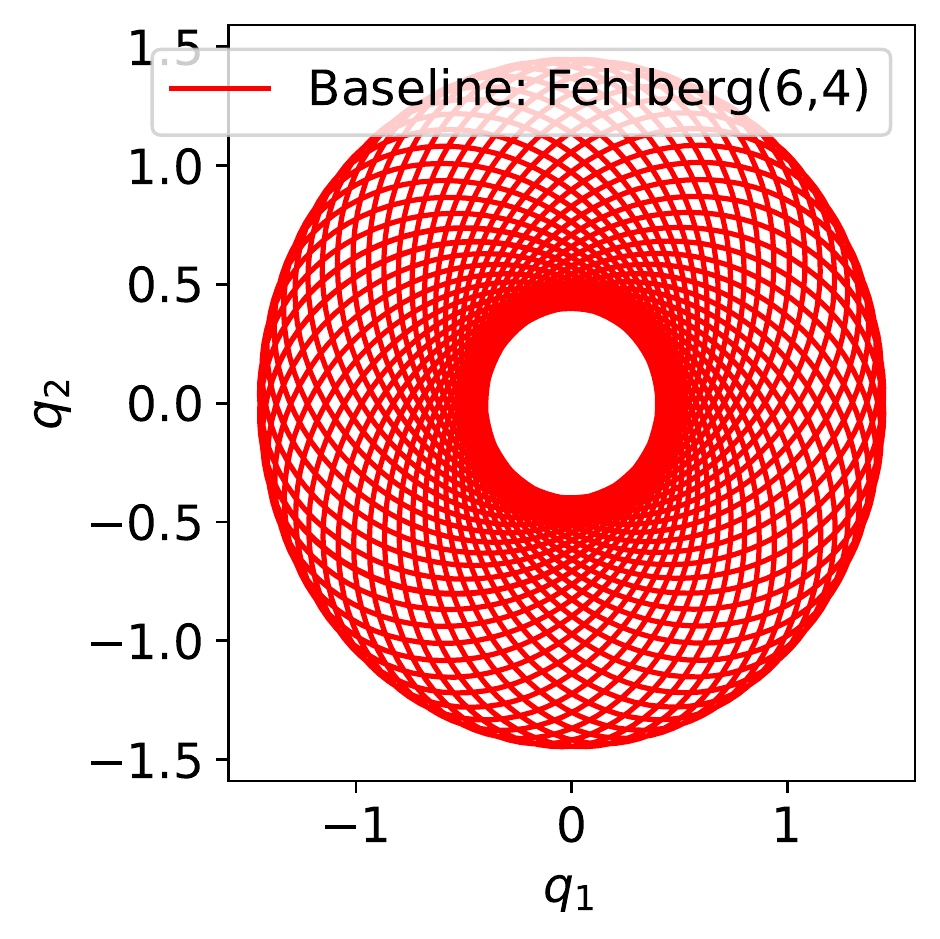}
    \end{subfigure}
    \hfill
    \begin{subfigure}[b]{0.32\textwidth}
        \centering
        \includegraphics[width=\textwidth]{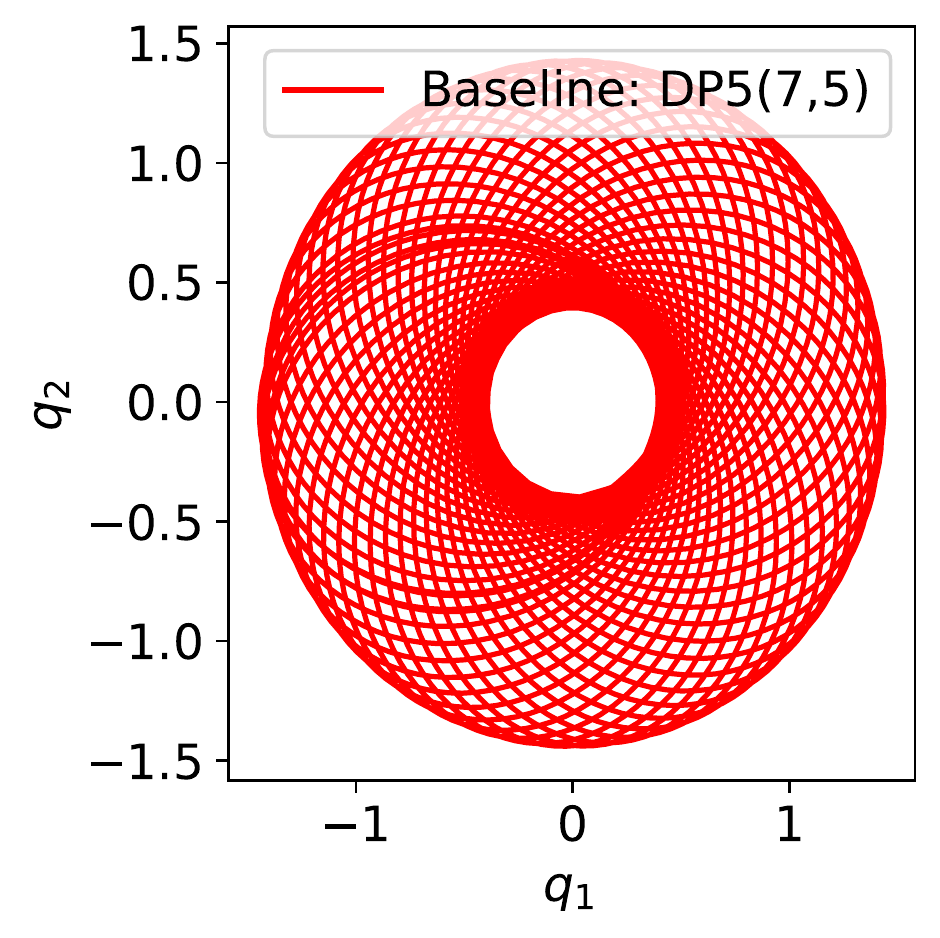}
    \end{subfigure}
    \begin{subfigure}[b]{0.32\textwidth}
        \centering
        \includegraphics[width=\textwidth]{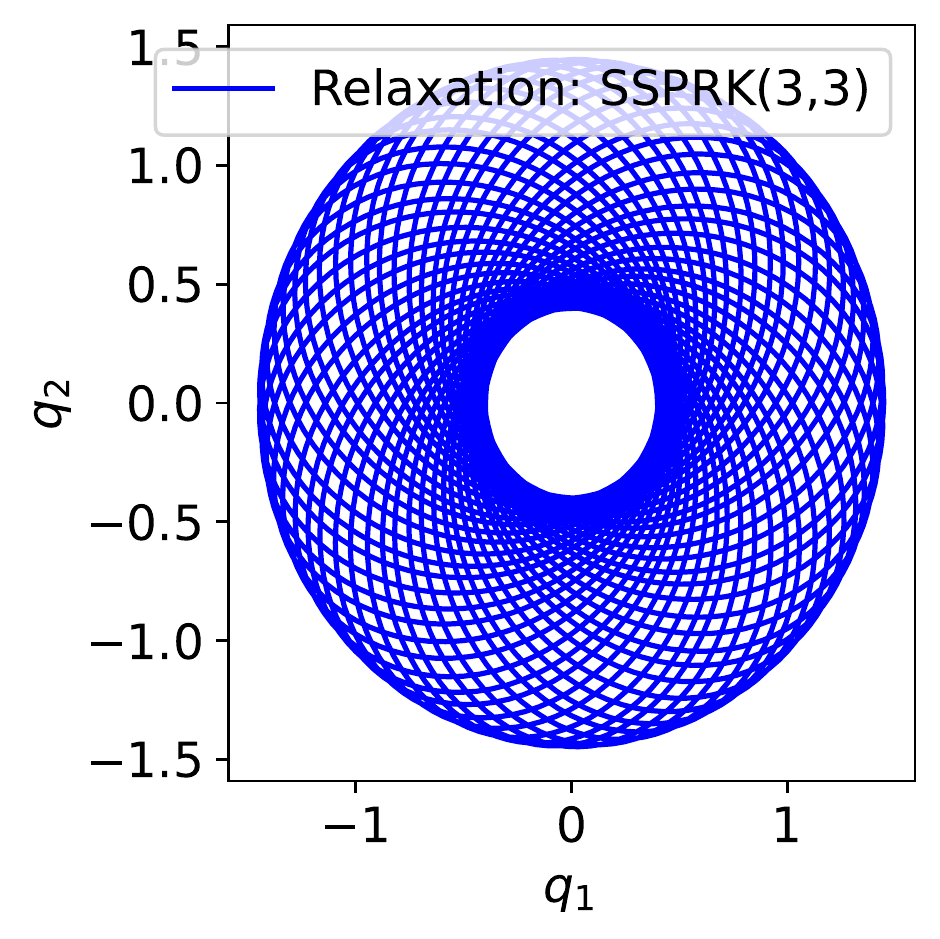}
    \end{subfigure}
    \hfill
    \begin{subfigure}[b]{0.32\textwidth}
        \centering
        \includegraphics[width=\textwidth]{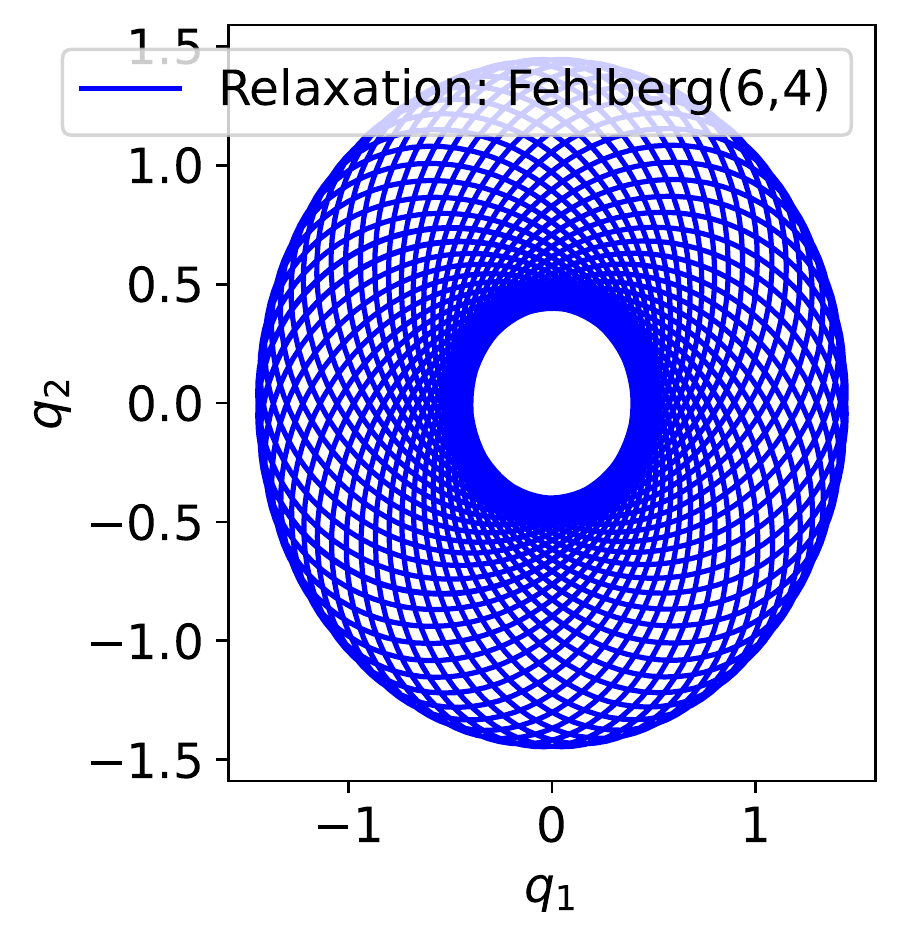}
    \end{subfigure}
    \hfill
    \begin{subfigure}[b]{0.32\textwidth}
        \centering
        \includegraphics[width=\textwidth]{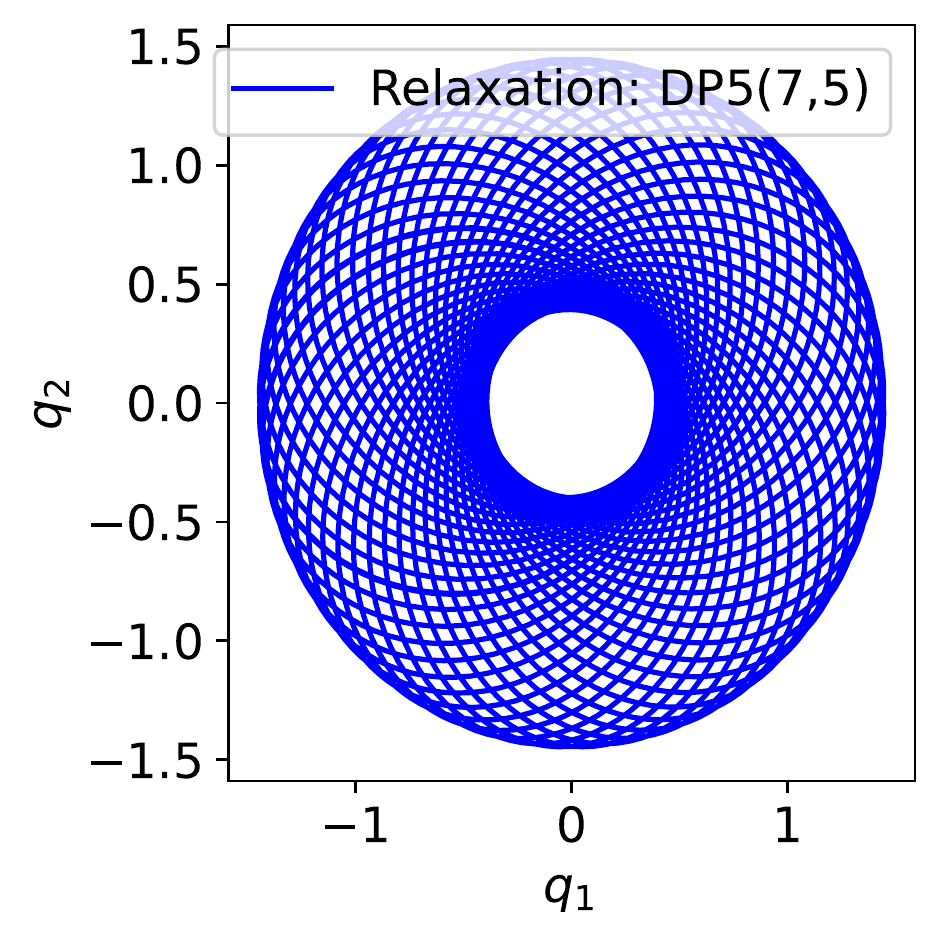}
    \end{subfigure}
    \hfill
    \begin{subfigure}[b]{0.24\textwidth}
    \centering
        \includegraphics[width=\textwidth]{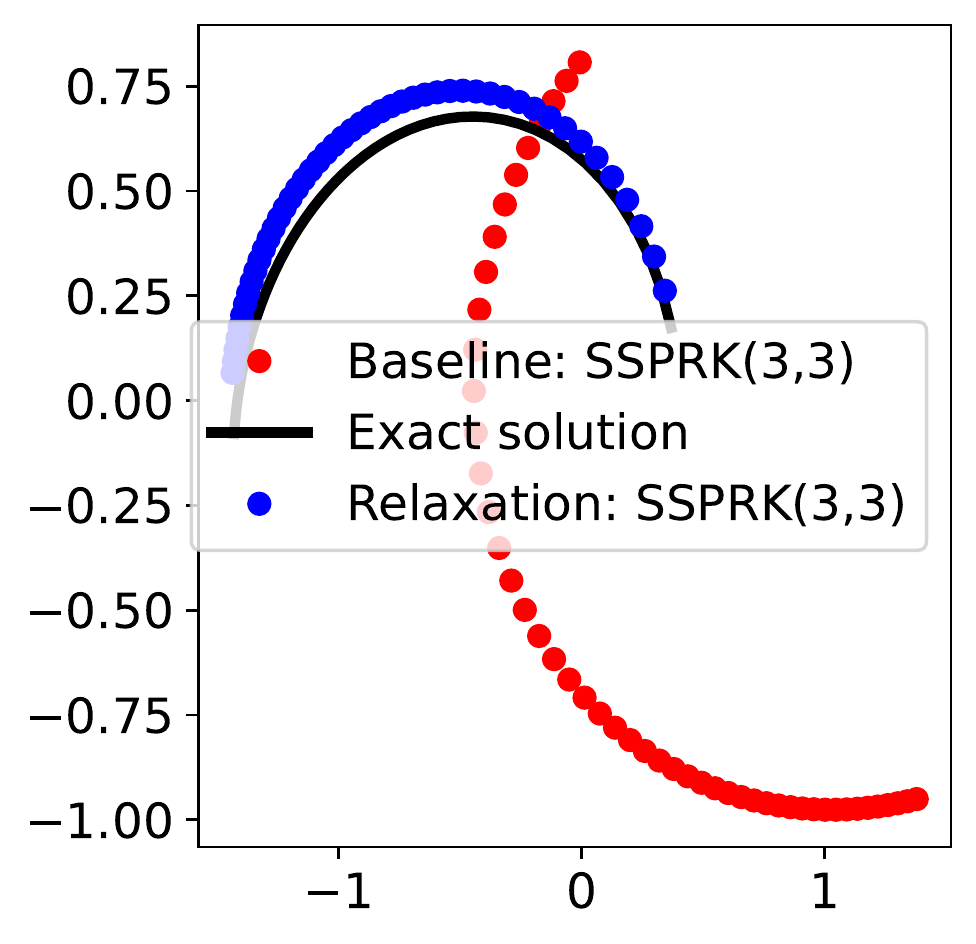}
    \end{subfigure}
    \begin{subfigure}[b]{0.24\textwidth}
    \centering
        \includegraphics[width=\textwidth]{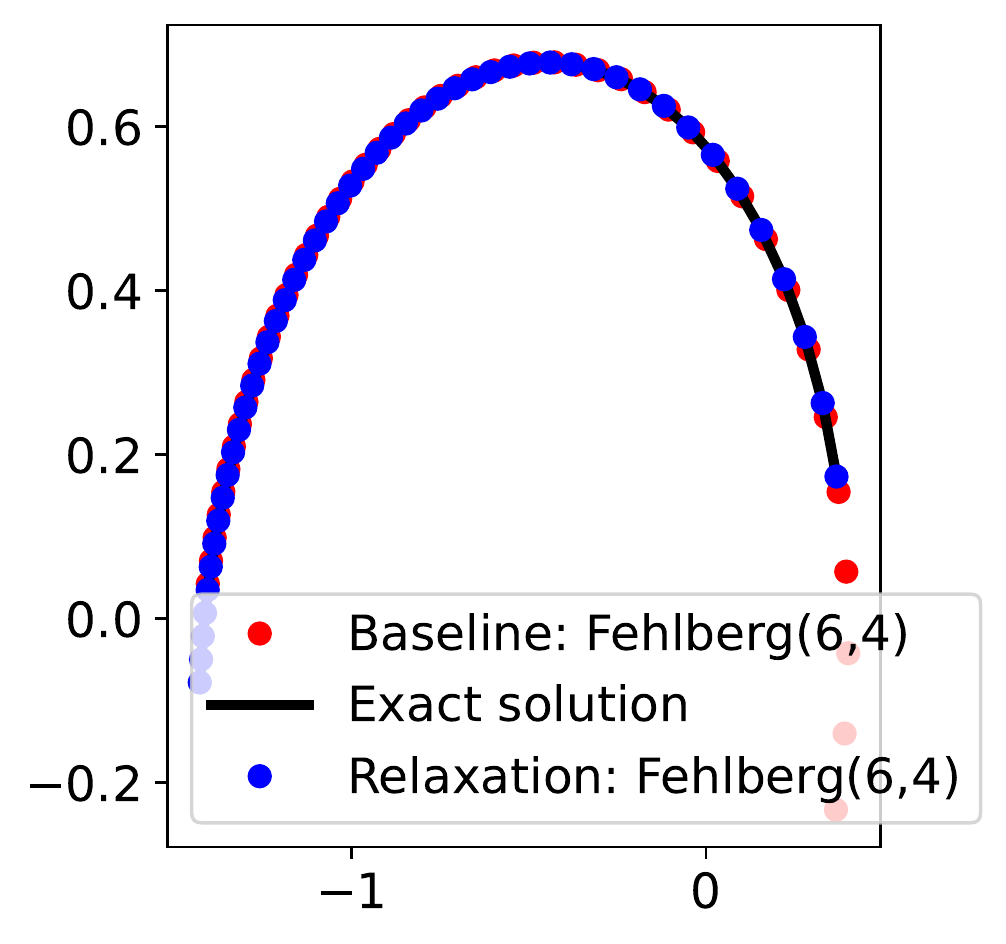}
    \end{subfigure}
    \begin{subfigure}[b]{0.24\textwidth}
    \centering
        \includegraphics[width=\textwidth]{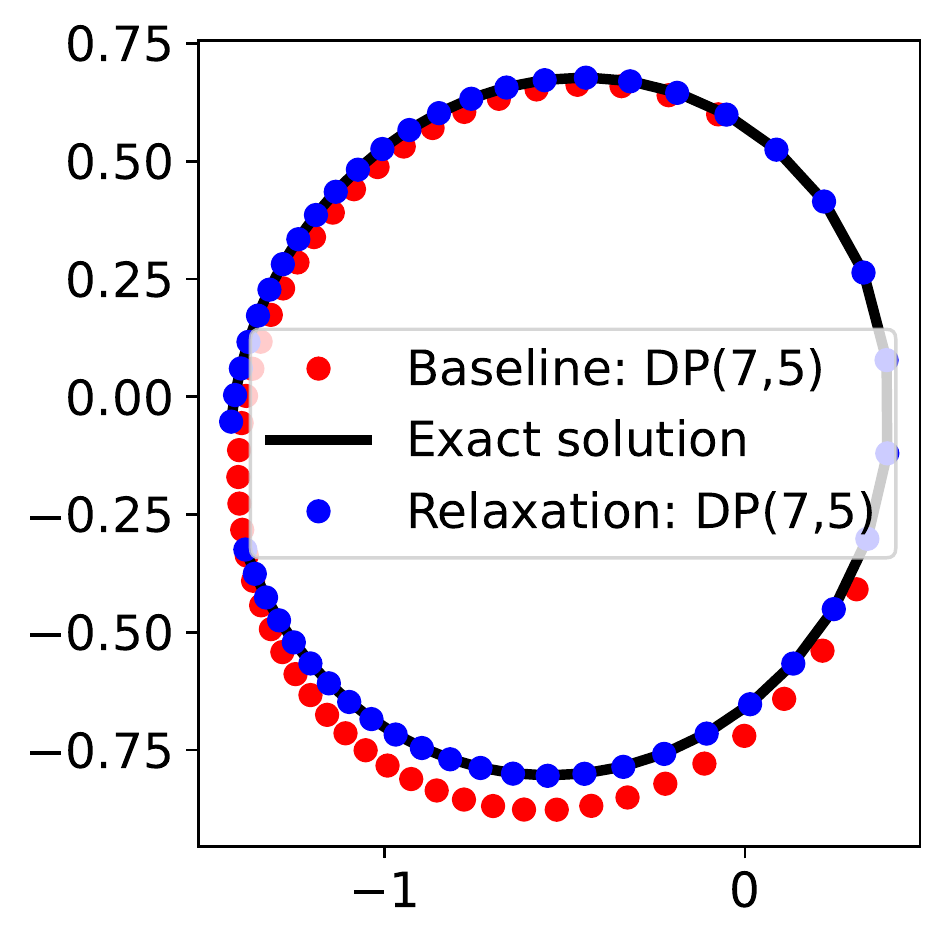}
    \end{subfigure}
    \begin{subfigure}[b]{0.24\textwidth}
    \centering
        \includegraphics[width=\textwidth]{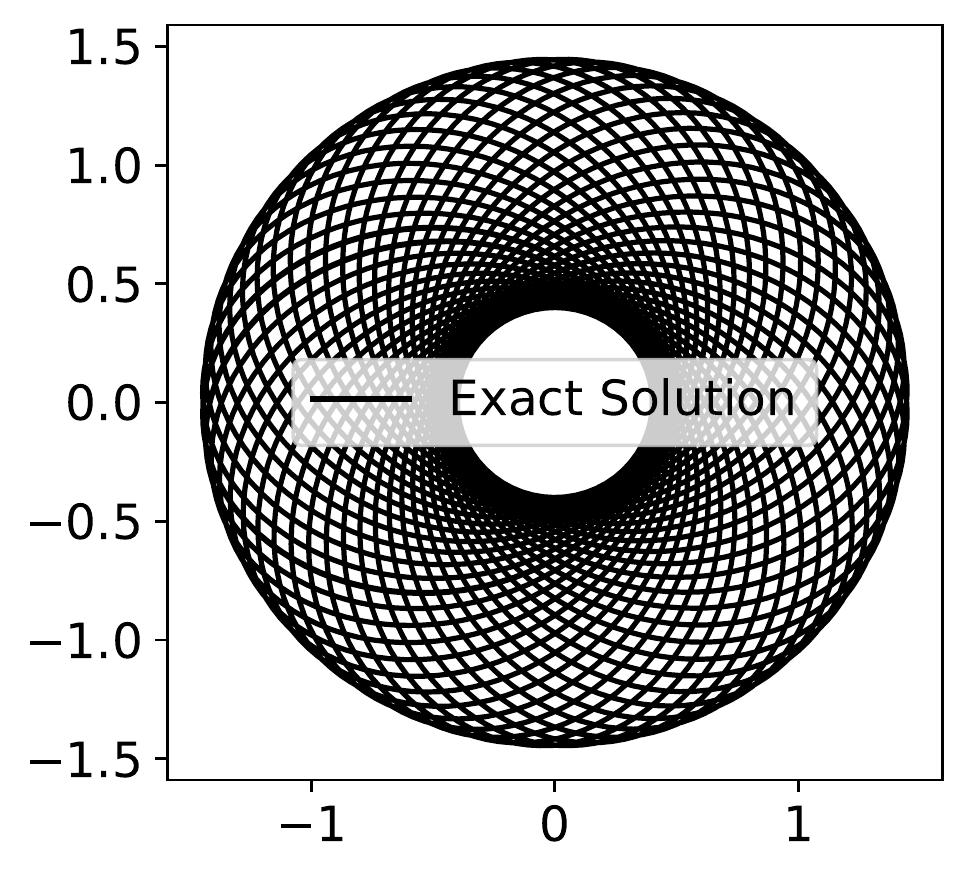}
    \end{subfigure}
    	\caption{Solution of perturbed Kepler's problem \eqref{Eq:Perturbed_Kepler_problem} by different methods. The first three plots in the last row of the figure compare the body's position by different methods at the few last steps of the simulation.}
    	\label{fig:Perturbed_Kepler_Solutions}
    \end{figure}
     Figure ~\ref{fig:Perturbed_Kepler_Evolution_conserved_quantities} shows that, in contrast with the underlying baseline methods, all the MRRK methods preserve the invariants up to the machine precision and correctly produce the elliptic orbit that precesses slowly around one of its foci (Figure ~\ref{fig:Perturbed_Kepler_Solutions}). Note that the effect of preserving the conserved quantities is only noticeable for the third-order method, where the baseline third-order method incorrectly captures the motion of orbits. Even though visually, the higher-order methods appear to produce the correct behavior of the trajectories without relaxation, in truth, these methods give completely wrong positions as time grows. As can be seen in Figure ~\ref{fig:Per_Kepler_error_growth}, the error of the baseline method increases quadratically with time until it reaches a saturation point of $100$\% error, leading to incorrect orbits of the body. The invariant-preserving relaxation approach, in contrast, results in a linear accumulation of error over time, which leads to a significantly smaller error, producing correct orbits of the body for a long time. 
     
     \begin{figure}
     \centering
         \includegraphics[width=\textwidth]{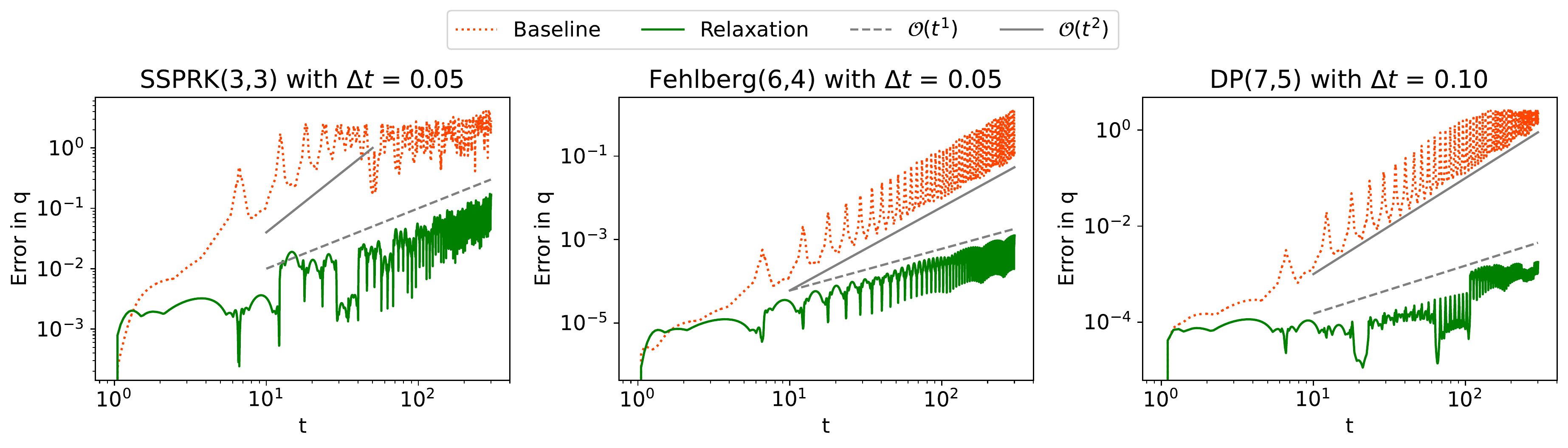}
    	\caption{Error growth over time for perturbed Kepler's problem \eqref{Eq:Perturbed_Kepler_problem}.}
    	\label{fig:Per_Kepler_error_growth}
    \end{figure}

\section{Application to the Korteweg–De Vries (KdV) equation}
\label{sec:KdV}
    Finally we consider a PDE example; namely,
    the Kortweg-de Vries (KdV) equation
    \begin{align}\label{KdV_Equation}
        u_t+6uu_x+u_{xxx} & = 0, \ \text{on} \ [x_{L},x_{R}]\times(0,T] \;,
    \end{align}
    with periodic boundary condition $u(x_L,t) = u(x_R,t)$.
    The KdV equation has infinitely many conserved quantities, of which the first three are
    \begin{subequations} \label{KdV_invariants}
        \begin{align}
            \int u \ dx \ (\text{mass}) \\
        	\int \frac{1}{2} u^2 \ dx \ (\text{energy}) \\
        	\int \left(2u^3 -u_x^2  \right) dx \ (\text{Whitham}) \;.
            \label{KdV_invariants_3}   
        \end{align}
    \end{subequations}
    It has been shown that, for the case of a 1-soliton solution, numerical methods that conserve both mass and energy give solutions whose error grows linearly in time, whereas methods that don't conserve these quantities generically yield quadratic error growth \cite{de1997accuracy}.
    In the same work, preliminary experiments with conservative methods applied to two interacting solitons also exhibited linear error growth, except during the soliton interaction.  However, there are no theoretical results except in the 1-soliton case.

    In this section we investigate the effect of conserving only the mass (which is conserved automatically by even the baseline RK methods) versus using relaxation to conserve both mass and energy, or  all three invariants \eqref{KdV_invariants}.  We consider initial data with one, two, or three solitons.
    We use relaxation to enforce the conservation of the nonlinear invariants.

    To discretize in space, we introduce an evenly-spaced grid ${x_i}$ with $x_1=x_L$ and $x_N=x_R-\Delta x$.
    We enforce semi-discrete mass and energy conservation by employing the split form spatial discretization \cite{ranocha2006broad}
    \begin{equation}\label{SBP_semi_discretization}
       \partial_{t} U+6 \times \frac{1}{3}\left(D_{1}(U \cdot U)+U \cdot\left(D_{1} U\right)\right)+D_{3} U=0,
    \end{equation}
    where $D_1, D_3$ are skew-symmetric differentiation matrices approximating the first- and third-derivative operators, respectively (here we use Fourier spectral differentiation matrices),
    and $U_i(t) \approx u(x_i,t)$.  
    In \eqref{SBP_semi_discretization} the dot $\cdot$ denotes element-wise multiplication.
    This guarantees semi-discrete conservation of the discrete mass and energy
        \begin{subequations} \label{KdV_discrete_inv_mass_energy}
        \begin{align}
            \eta_0 & := \Delta x \sum_{j}U_j \;, \\
            \eta_1 & := \Delta x \sum_{j}|U_j|^2,
        \end{align}
    \end{subequations}
    (up to rounding errors), regardless of the grid spacing $\Delta x$.  However, the third invariant is conserved only to the level of the spatial truncation error.  This can be made small by using a fine grid, and will be discussed further in Section \ref{sec:conserve3}.

    The semi-discretization \eqref{SBP_semi_discretization} is stiff, since $\|D_3\| = {\mathcal O}(\Delta x^{-3})$.  We therefore make use of ImEx RK schemes in time, handling the stiff linear term implicitly and the nonlinear term explicitly.  We use two ImEx schemes from \cite[Appendix C]{kennedy2003additive}, which were introduced already at the beginning of Section \ref{sec:numerical}.
    
    \subsection{Conservation of mass and energy}
    In this section we investigate how the conservation of mass and energy affects the temporal error growth, relative to methods that conserve only the mass.  We consider 3 different initial conditions, consisting of one, two, or three solitons, as detailed in Appendix \ref{app:soliton-solutions}.

    \subsubsection{One soliton}
    We first consider the 1-soliton solution ((a) in appendix \ref{app:soliton-solutions}) on the domain $[x_{L},x_{R}]=[-20,60]$ of length $L = 80$ and $N=512$ spatial grid points and integrate from $t=0$ to $t=20$. Table ~\ref{Table:1_invariant_error} displays the maximum deviation of each invariant compared to its initial value, confirming that mass is conserved in all cases, while energy is conserved only by the MRRK methods.
    It is interesting to note that enforcing conservation of $\eta_1$ also greatly reduces the amount of variation in the third invariant \eqref{KdV_invariants_3}, shown in the last column (the discrete approximation used to compute this invariant is given in \eqref{discrete_Whitham}).
    In Figure ~\ref{fig:KdV_1_sol_1_inv_GlobalErrors}, we plot the global error as a function of time. The errors behave linearly for RK methods with relaxation and quadratically without relaxation. These numerical results agree with the analytical results presented in \cite{de1997accuracy}, i.e., in the case of 1-soliton solution of the KdV equation, the errors incurred by methods conserving mass and energy grow linearly as opposed to quadratic growth by non-conservative methods.
    \begin{table}
    \centering
    \resizebox{15cm}{!}{%
    \begin{tabular}{ | c | l | c | c | c |}
        \hline
               &    &          \multicolumn{3}{c|}{Maximum changes in invariants}\\
        \hline
        \parbox{10em}{Mass and energy \\ conservative semi- \\ discretization with} & Methods & Mass & Energy & Whitham  \\
        \hline
        \multirow{4}*{One soliton} & Baseline ARK3(2)4L[2]SA  & $1.33e-15$ & $5.38e-02$ & $2.11e-01$ \\
        \cline{2-5}
                             & Relaxation ARK3(2)4L[2]SA  & $2.22e-15$ & $1.33e-15$ &  $6.56e-04$ \\
        \cline{2-5}
                             & Baseline ARK4(3)6L[2]SA  & $8.88e-16$ & $1.05e-02$ & $4.21e-02$ \\
        \cline{2-5}
                             & Relaxation ARK4(3)6L[2]SA  & $8.88e-16$ & $1.55e-15$ & $9.84e-05$ \\
        \hline
        \multirow{4}*{Two soliton} & Baseline ARK3(2)4L[2]SA  & $2.66e-15$ & $1.10e-01$ & $4.20e-01$ \\
        \cline{2-5}
                             & Relaxation ARK3(2)4L[2]SA  & $2.66e-15$ & $2.66e-15$ & $7.40e-03$\\
        \cline{2-5}
                             & Baseline ARK4(3)6L[2]SA  & $2.66e-15$ & $2.32e-02$ & $9.16e-02$ \\
        \cline{2-5}
                             & Relaxation ARK4(3)6L[2]SA  & $2.66e-15$ & $5.33e-15$ &  $1.92e-03$\\
        \hline 
        \multirow{4}*{Three soliton} & Baseline ARK3(2)4L[2]SA  & $2.66e-15$ & $2.07e-01$ & $7.45e-01$  \\
        \cline{2-5}
                             & Relaxation ARK3(2)4L[2]SA  & $6.22e-15$ & $4.22e-15$ & $3.10e-02$  \\
        \cline{2-5}
                             & Baseline ARK4(3)6L[2]SA  & $3.55e-15$ & $4.62e-02$ & $1.79e-01$ \\
        \cline{2-5}
                             & Relaxation ARK4(3)6L[2]SA  & $3.55e-15$ & $3.55e-15$ &  $9.64e-03$\\
        \hline
    \end{tabular}%
    }
    \caption{Maximum changes in invariants \eqref{KdV_discrete_inv_mass_energy} and \eqref{discrete_Whitham} by different methods applied to a mass and energy conserving semi-discretized KdV equation with different n-solitons.}
    \label{Table:1_invariant_error}
    \end{table}
    
    \begin{figure}
    \centering
    \includegraphics[width=\textwidth]{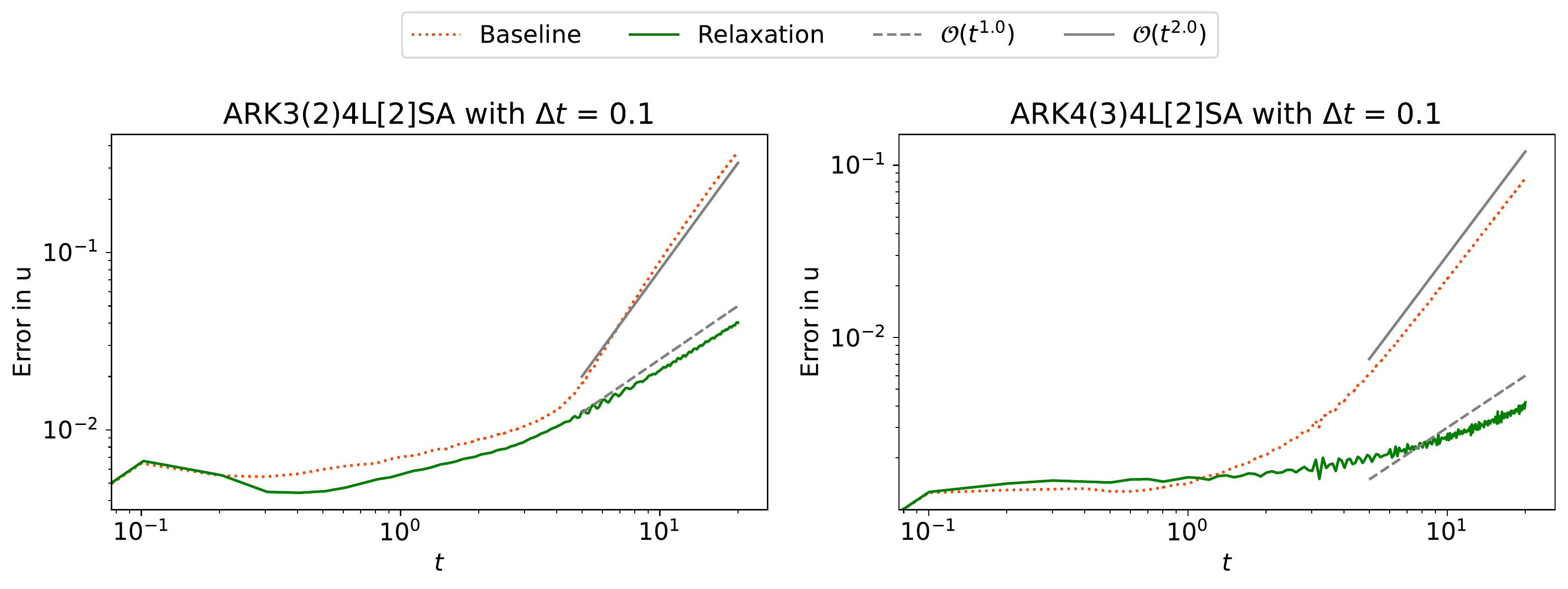}
	\caption{Error growth over time for a 1-soliton solution of the KdV equation.  Relaxation is used to enforce conservation of $\eta_1$.}
	\label{fig:KdV_1_sol_1_inv_GlobalErrors}
    \end{figure}
    
    \subsubsection{Two solitons}
    Next, we consider a 2-soliton solution over the region $[x_{L},x_{R}]=[-80,80]$ and use the semi-discretization \eqref{SBP_semi_discretization} with $N=1024$ spatial grid points, integrating from $t=-25$ to $t=25$. 
    The error growth for the resulting solutions is presented in Figure ~\ref{fig:KdV_2_sol_1_inv_GlobalErrors}.  In this case, there are no available theoretical results to guarantee the error growth of conservative methods will be better than for non-conservative methods.  For the two relaxation methods employed here we see markedly different behaviors; one method exhibits sublinear growth, while the other exhibits something between linear and quadratic growth at long times.  In both cases, the conservative (relaxation) methods provide solutions that are drastically more accurate compared to the non-conservative counterparts, which exhibit the expected quadratic error growth.  All methods exhibit a dip in the error during the soliton interaction, as was observed in \cite{de1997accuracy}.
     
     \subsubsection{Three solitons}
     Finally, we consider the case of a 3-soliton solution on the domain $[x_{L},x_{R}]=[-130,130]$ with $N=1536$ spatial grid points, integrated from $t=-50$ to $t=50$. Figure ~\ref{fig:KdV_3_sol_1_inv_GlobalErrors} shows the errors over time. In this case, the conservative (relaxation) approach results in quadratic growth of errors similar to baseline ImEx methods, although the conservative solutions still have much smaller errors.  We examined the structure of the errors in this case and found that although the total energy is conserved, the energy of individual solitons changes linearly over time, leading to  quadratically growing phase errors for all three solitons.  It is unclear why the 2-soliton case does not exhibit a similar effect.
     
    \begin{figure}
    \centering
    \includegraphics[width=\textwidth]{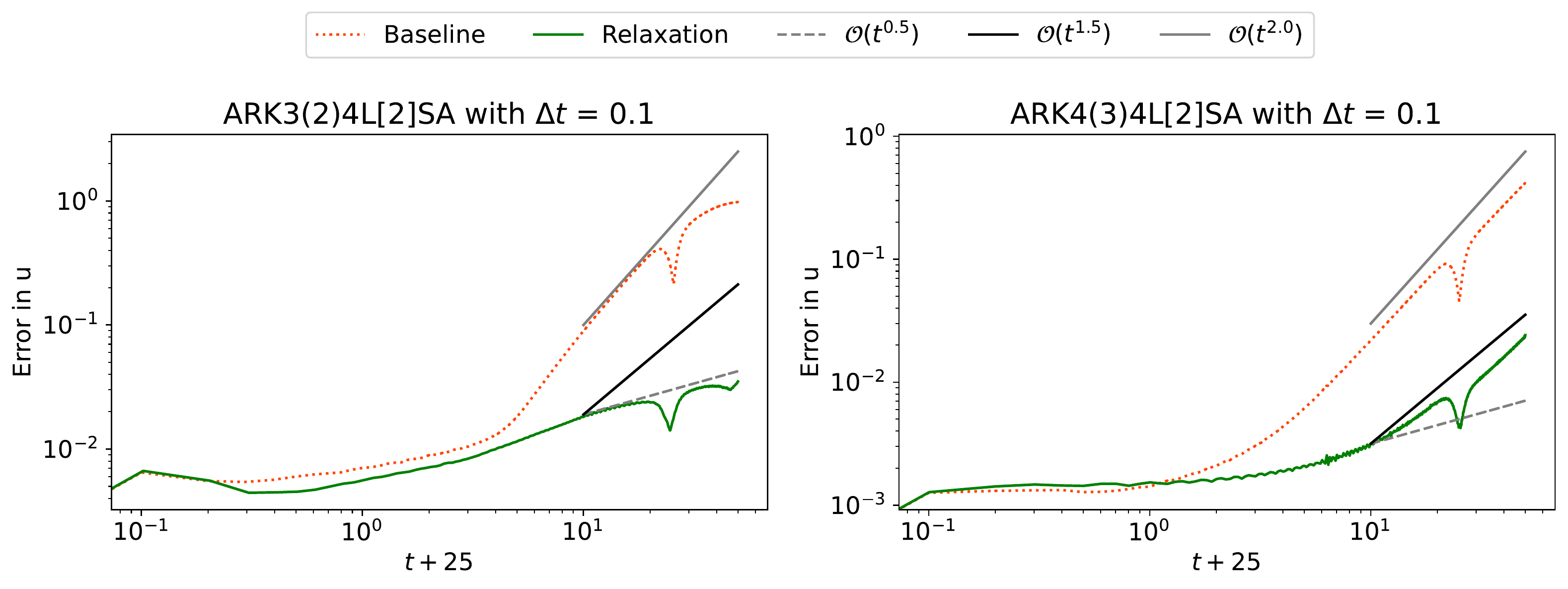}
	\caption{Error growth over time for a 2-soliton solution of the KdV equation.  Relaxation is used to enforce conservation of $\eta_1$.}
	\label{fig:KdV_2_sol_1_inv_GlobalErrors}
    \end{figure}
    
    \begin{figure}
    \centering
    \includegraphics[width=\textwidth]{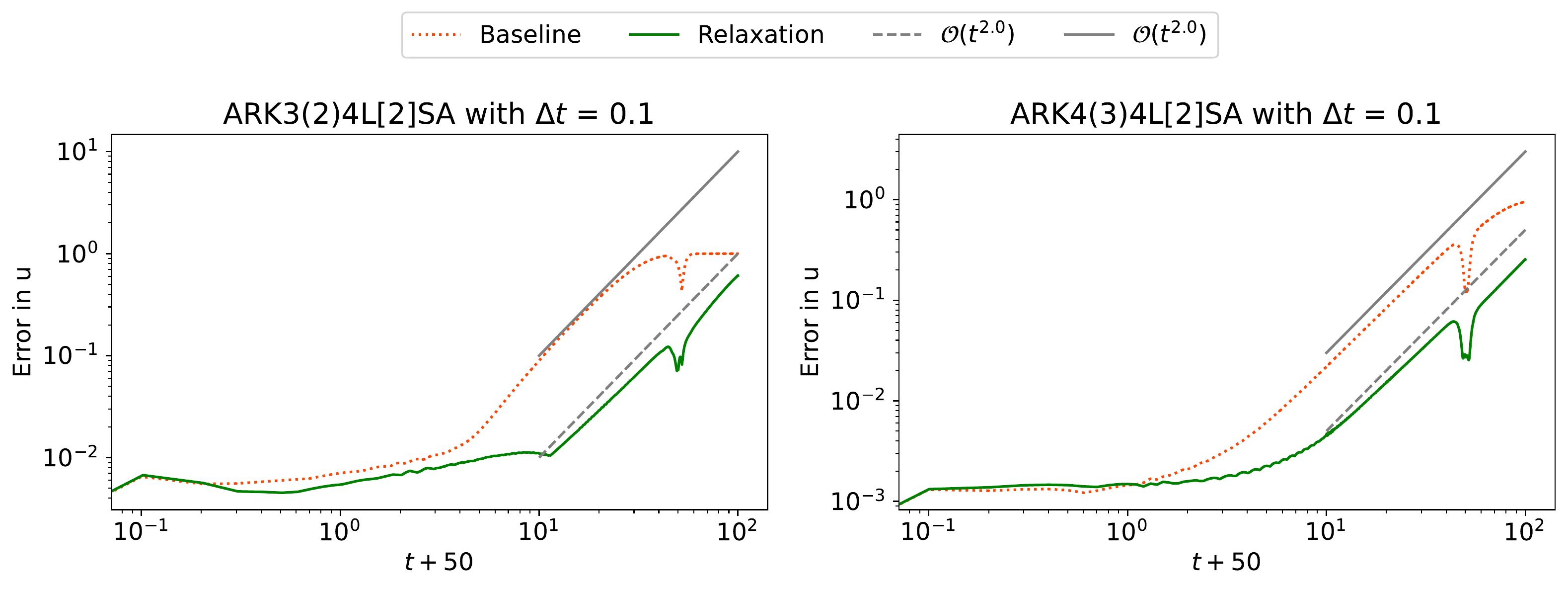}
	\caption{Error growth over time for a 3-soliton solution of the KdV equation.  Relaxation is used to enforce conservation of $\eta_1$.}
	\label{fig:KdV_3_sol_1_inv_GlobalErrors}
    \end{figure}
   
     \subsection{Multiple relaxation: attempting to restore conservation through relaxation}\label{sec:conserve3}
    We have seen that conserving only two invariant quantities (mass and energy) is not generally sufficient to produce linear error growth for the 2- and 3-soliton solutions.  It is natural to ask whether conserving a third invariant will improve this. 
    
    Since we do not have a semi-discretization that preserves discrete analogs of all three invariants simultaneously, we instead attempt to use a sufficiently fine spatial grid (still with the pseudospectral spatial discretization \eqref{SBP_semi_discretization}) in order to make the overall spatial error small, so that the semi-discrete error in the conservation of $\eta_2$ will also be as small as possible.
     We introduce the discrete approximation to the third invariant
    \begin{align}\label{discrete_Whitham}
        \eta_2 := 2\frac{\Delta x}{3}\left[U_{1}^{3}+\sum_{j=2}^{N-1}\left(4U_{j}^{3}+2U_{j+1}^{3}\right)+U_{N+1}^{3}\right]-\frac{\Delta x}{3}\left[{V}_{1}^{2}+\sum_{j=2}^{N-1}\left(4{V}_{j}^{2}+2{V}_{j+1}^{2}\right)+{V}_{N+1}^{2}\right] \;,
    \end{align}
    where $V:=D_1 U \approx u_x$, the numerical derivative is computed using Fourier transformation, and the integrals $\int u^3 dx$ and $\int u_{x}^{2} dx$ are approximated using Simpson's quadrature rule applied to the function $U^3$ and $(D_1 U)^{2}$, respectively. We test how well the exact solution conserves this quantity as we refine the spatial grid, by computing
    $$
    \max_t\|\eta_2'(t)\|_\infty
    $$
    for the exact solution over the time interval of interest.
    The minimum achievable value is about $10^{-11}$, which is obtained with
    $N = 1024$ grid points for a 2-soliton solution on the domain $[-80,80]$ and $N = 1536$ grid points for a 3-soliton solution on the domain $[-130,130]$.
    

    \subsubsection{Solution of the relaxation equations for a non-conservative system}
    Since $\eta_2$ is not exactly conserved by the true semi-discrete solution, we have no theoretical guarantee of the existence of a solution of the relaxation equations \eqref{relaxation-equations}.  In fact, this represents an interesting potential application of the relaxation technique -- if conservation is lost in the process of semi-discretization, can we restore it through time discretization, and what effect will that have?  The results below represent a first exploration of this question, which might serve as a starting point for further work.

    We find that at many timesteps, the \verb|fsolve| routine from \verb|scipy.optimize| gives a solution
    accurate to significantly less than double precision, consistent with our estimates of the accuracy of conservation for the semi-discrete scheme.  In order to ensure that we obtain the most accurate solution possible, we do an additional search using numerical optimization routines from \verb|scipy.optimize| when the solution from \verb|fsolve| is inaccurate.

    \subsubsection{Results}
    Figures ~\ref{fig:KdV_2_sol_2_inv_InvariantErrors} and ~\ref{fig:KdV_3_sol_2_inv_InvariantErrors} show the deviation of the invariants over time for the 2-soliton and 3-soliton cases, respectively.  Relaxation methods improve the error in invariants, but they are not up to the machine accuracy.

    \begin{figure}
    \centering
    \includegraphics[width=\textwidth]{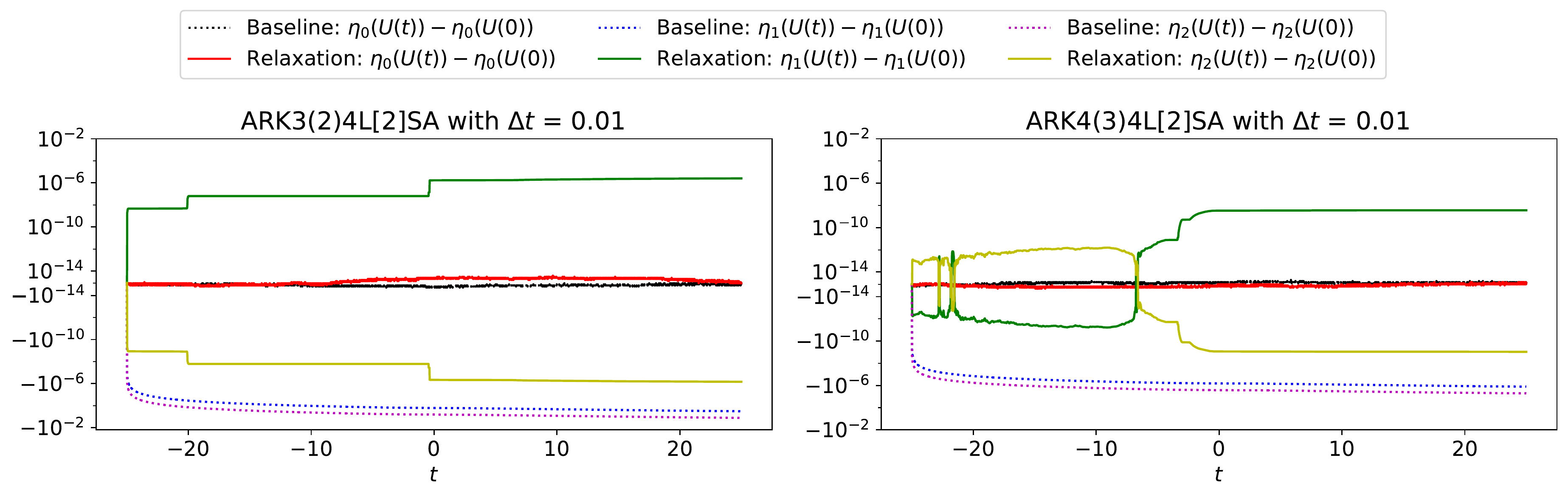}
	\caption{Deviation of each invariant from its initial value for a 2-soliton solution of the KdV equation.  Relaxation is used to (attempt to) enforce conservation of $\eta_1$ and $\eta_2$.}
	\label{fig:KdV_2_sol_2_inv_InvariantErrors}
    \end{figure}

    \begin{figure}
    \centering
    \includegraphics[width=\textwidth]{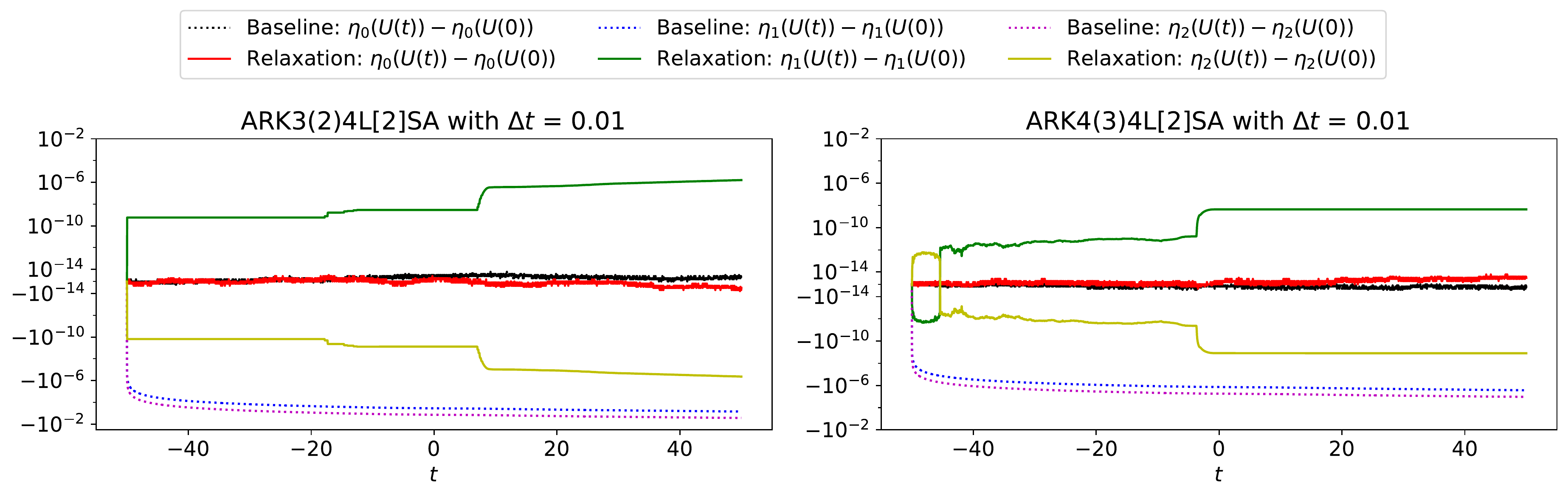}
	\caption{Deviation of each invariant from its initial value for a 3-soliton solution of the KdV equation.  Relaxation is used to (attempt to) enforce conservation of $\eta_1$ and $\eta_2$.}
	\label{fig:KdV_3_sol_2_inv_InvariantErrors}
    \end{figure}
    The error growth behavior for these examples is plotted in Figures \ref{fig:KdV_2_sol_2_inv_GlobalErrors} and \ref{fig:KdV_3_sol_2_inv_GlobalErrors}.  We observe that relaxation may yield no improvement or even degrade the accuracy of the solution over short times, but provides a noticeable improvement in accuracy over sufficiently long times.  Interestingly, we observe roughly linear error growth at long times, although there is significant jitter, probably due to the lack of an exact solution to the relaxation equations.
    
    \begin{figure}
    \centering
    \includegraphics[width=\textwidth]{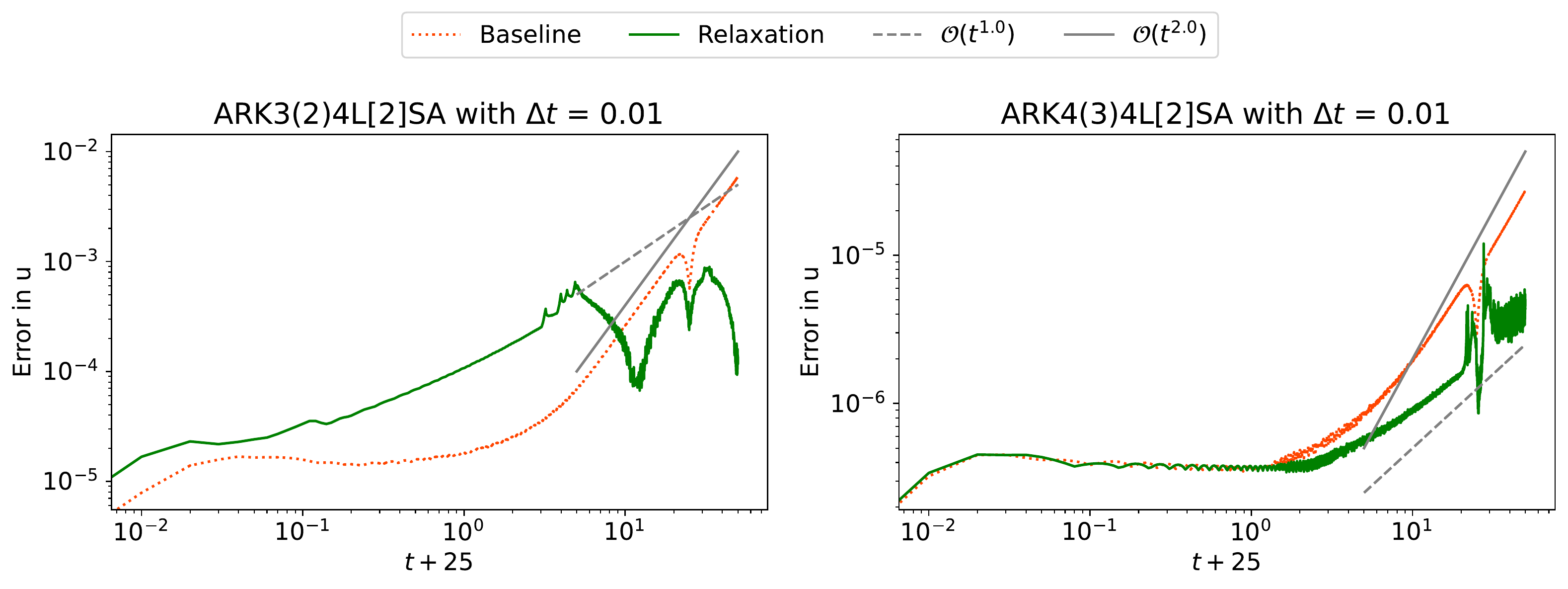}
	\caption{Error growth over time for a 2-soliton solution of the KdV equation.  Relaxation is used to enforce conservation of $\eta_1$ and $\eta_2$ as nearly as possible.}
	\label{fig:KdV_2_sol_2_inv_GlobalErrors}
    \end{figure}

    \begin{figure}
    \centering
    \includegraphics[width=\textwidth]{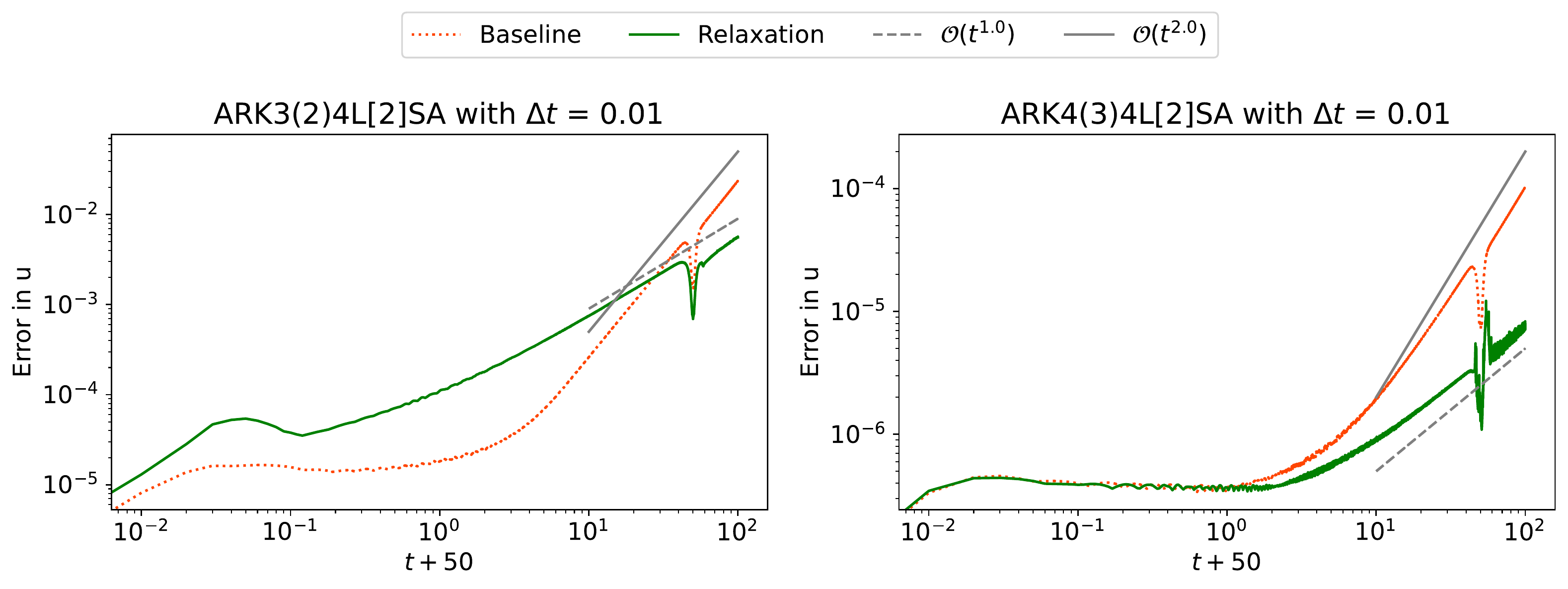}
	\caption{Error growth over time for a 3-soliton solution of the KdV equation.  Relaxation is used to enforce conservation of $\eta_1$ and $\eta_2$ as nearly as possible.}
	\label{fig:KdV_3_sol_2_inv_GlobalErrors}
    \end{figure}

\section{Conclusions}
    In this work, we propose a generalization of the relaxation framework for RK methods to preserve multiple nonlinear conserved quantities of a dynamical system. We prove the existence of the relaxation parameters and the accuracy of the generalized relaxation methods under some conditions. 
    We also demonstrate for the first time the application of relaxation in combination with additive (ImEx) RK methods for stiff problems.
    Numerical results indicate that multiple-relaxation RK methods can conserve multiple conserved quantities and produce qualitatively better numerical solutions for conservative ODE and PDE dynamical systems. 

    An important case study of our numerical experiments is the KdV equation with multi-soliton solutions. With a conservative semi-discretization of the KdV equation preserving mass and energy, the relaxation approach applied with ARK methods successfully preserves the invariants for 1, 2, and 3-soliton solutions. We observe that mass and energy preservation do not necessarily guarantee linear error growth for multi-soliton solutions for the KdV equation. Some numerical results applying relaxation methods to a non-conservative semi-discretization (enforcing conservation of a third nonlinear quantity that is conserved by the PDE but not the spatial discretization) of the KdV equation are presented. Numerical results suggest that relaxation methods can be advantageous if we are interested in long time numerical solutions. 

    The application of this general relaxation framework to the nonlinear Schr\"{o}dinger equation with multiple nonlinear invariants is a subject of ongoing research. Another possible research direction is extending the generalized relaxation approach framework to the other class of time integration schemes, such as linear multistep methods, and improving the underlying methods' numerical performance.  Additionally, multiple relaxation could be applied to systems with multiple dissipated functionals, or with some conserved and some dissipated functionals.
\appendix
\section{List of RK Methods} 
\label{app:butcherTableau}

\begin{table}[H]
    \centering
    \begin{adjustbox}{angle=0}
    \resizebox{2.5cm}{!}{%
    	\begin{tabular}{l|l l}
             0  &    0 \\                        
             1  &    1  &   0  \\ 
             \hline
            $\vec{b}^{\,1}$  &   1/2  &   1/2 \\
             \hline
            $\vec{b}^{\,2}$  &   1/3  &   2/3	 				   			 
    	\end{tabular}%
    	}%
    \end{adjustbox}
    \caption{SSPRK(2,2): Second-order method $(A,\vec{b}^{\,1})$ with a first-order embedded method $(A,\vec{b}^{\,2})$.}
    \label{tab:BT_ssprk22}
\end{table}

\begin{table}[H]
    \centering
    \begin{adjustbox}{angle=0}
    \resizebox{10cm}{!}{%
    	\begin{tabular}{l|l l l}
             0  &    0 \\                        
             1  &    1  &   0  \\ 
             1/2 &          1/4   &        1/4 & 0 \\ 
             \hline
            $\vec{b}^{\,1}$  &   1/6  &          1/6    &        2/3   \\
             \hline
            $\vec{b}^{\,2}$  &   0.291485418878409 & 0.291485418878409 & 0.417029162243181	\\
            \hline
            $\vec{b}^{\,3}$  & 0.395011932394815 & 0.395011932394815 & 0.209976135210371
    	\end{tabular}%
    	}%
    \end{adjustbox}
    \caption{SSPRK(3,3): Third-order method $(A,\vec{b}^{\,1})$ with $(A,\vec{b}^{\,2})$ and $(A,\vec{b}^{\,3})$ as second-order embedded methods.}
    \label{tab:BT_ssprk33}
\end{table}

\begin{table}[H]
    \centering
    \begin{adjustbox}{angle=0}
    \resizebox{10cm}{!}{%
    	\begin{tabular}{l|l l l}
             0  &    0 \\                        
             1/3  &    1/3  &   0  \\ 
             2/3 &      0   &    2/3 & 0 \\ 
             \hline
            $\vec{b}^{\,1}$  &   1/4  &          0   &        1/4   \\
             \hline
            $\vec{b}^{\,2}$  &   0.006419303047187  & 0.487161393905626 &  0.506419303047187	 				   			 
    	\end{tabular}%
    	}%
    \end{adjustbox}
    \caption{Heun(3,3): Third-order method $(A,\vec{b}^{\,1})$ with a second-order embedded method $(A,\vec{b}^{\,2})$.}
    \label{tab:BT_heun33}
\end{table}

\begin{table}[H]
    \centering
    \begin{adjustbox}{angle=0}
    \resizebox{4cm}{!}{%
    	\begin{tabular}{l|l l l l}
             0  &    0 \\                        
             1/2  &    1/2  &   0  \\ 
             1/2 &      0   &    1/2 & 0 \\ 
             1   &      0   &    0   &     1      &    0  \\
             \hline
            $\vec{b}^{\,1}$  &   1/6   &        1/3    &       1/3      &     1/6  \\
             \hline
            $\vec{b}^{\,2}$  &   1/4   &        1/4      &     1/4     &      1/4	 				   			 
    	\end{tabular}%
    	}%
    \end{adjustbox}
    \caption{RK(4,4): Fourth-order method $(A,\vec{b}^{\,1})$ with a second-order embedded method $(A,\vec{b}^{\,2})$.}
    \label{tab:BT_rk44}
\end{table}

\begin{table}[H]
    \centering
    \begin{adjustbox}{angle=0}
    \resizebox{18cm}{!}{%
    	\begin{tabular}{l|l l l l l l}
             0  &    0 \\                        
             1/4 &          1/4  &          0   \\ 
             3/8 &	3/32 &	9/32 & 0 \\ 
             12/13	& 1932/2197 & 	-7200/2197 &	7296/2197 &   0\\
             	1 &	439/216 &	-8 &	3680/513 &	-845/4104 &    0\\
              1/2 &	-8/27  &	2 &	-3544/2565 &	1859/4104 &	-11/40 &	    0\\
             \hline
             $\vec{b}^{\,1}$  & 16/135	& 0	& 6656/12825& 28561/56430	& -9/50 &	2/55 \\
             \hline
            $\vec{b}^{\,2}$  &  25/216 &	0 &	1408/2565 &	2197/4104  &	-1/5  &	0   \\
             \hline
            $\vec{b}^{\,3}$  &   0.122702088570621 &   0.000000000000003 &  0.251243531398616 & -0.072328563385151 &  0.246714063515406 &  0.451668879900505 \\
            \hline
            $\vec{b}^{\,4}$  & 0.150593325320835 & 0.000000000000003 & 0.275657325006399 & 0.414789231909538 & -0.131467847351019 & 0.290427965114243
    	\end{tabular}%
    	}%
    \end{adjustbox}
    \caption{Fehlberg(6,4): Fifth-order method $(A,\vec{b}^{\,1})$ with a fourth-order embedded method $(A,\vec{b}^{\,2})$ , and $(A,\vec{b}^{\,3})$ and $(A,\vec{b}^{\,4})$ as third-order embedded methods.}
    \label{tab:BT_fehlberg54}
\end{table}

\begin{table}[H]
    \centering
    \begin{adjustbox}{angle=0}
    \resizebox{17cm}{!}{%
    	\begin{tabular}{l|l l l l l l l}
             0  &    0 \\                        
             1/5 &          1/5  &          0   \\ 
            3/10 & 	3/40 & 	9/40 & 0 \\ 
             4/5 & 	44/45 & 	-56/15 & 	32/9 &   0\\
            8/9	&  19372/6561 & 	-25360/2187 & 	64448/6561 & 	-212/729 &    0\\
            1	& 9017/3168  & 	-355/33	 &  46732/5247 & 	49/176 & 	-5103/18656 &	    0\\
            1	&  35/384  & 	0  & 	500/1113  & 	125/192 & 	-2187/6784 & 	11/84  &    0\\
             \hline
            $\vec{b}^{\,1}$  &  35/384 & 	0  & 	500/1113  & 	125/192  & 	-2187/6784  & 	11/84  & 	0   \\
             \hline
            $\vec{b}^{\,2}$  &   5179/57600  & 	0  & 	7571/16695  & 	393/640  & 	-92097/339200  &  187/2100  & 1/40	\\
            \hline
            $\vec{b}^{\,3}$  & 0.159422044716717 & 0.000000000000009 & 0.310936711045800 & 0.444052776789396 & 0.307005319740028 & -0.230738637667449 & 0.009321785375499
    	\end{tabular}%
    	}%
    \end{adjustbox}
    \caption{DP(7,5): Fifth-order method $(A,\vec{b}^{\,1})$ with $(A,\vec{b}^{\,2})$ and $(A,\vec{b}^{\,3})$ as fourth-order embedded methods.}
    \label{tab:BT_dp54}
\end{table}
\section{Soliton Solutions} 
\label{app:soliton-solutions}
Different soliton solutions \cite{ranocha2020relaxation} of the KdV equation \eqref{KdV_Equation} are given below: 
\begin{enumerate}
    \item [(a)] \textbf{1-soliton solution}:
    \begin{align}
        u(x,t)= \beta_1 \sech ^2(\xi_1) \;,
    \end{align}
    where $\beta_1 = 1$ and $\xi_1 = \frac{\sqrt{\beta_1}(x-2 \beta_1 t)}{\sqrt{2}}$.
    \item [(b)] \textbf{2-soliton solution}:
    \begin{align}
        u(x, t) = -\frac{2(\beta_1-\beta_2)\left(\beta_2 \csch^2(\xi_2)+\beta_1 \sech ^2(\xi_1)\right)}{\left(\sqrt{2 \beta_1} \tanh(\xi_1)-\sqrt{2 \beta_2} \coth(\xi_2)\right)^2} \;,
    \end{align}
    where $\beta_1 = 0.5$, $\beta_2 = 1$, $\xi_1 = \frac{\sqrt{\beta_1}(x-2 \beta_1 t)}{\sqrt{2}}$, and $\xi_2 = \frac{\sqrt{\beta_2}(x-2 \beta_2 t)}{\sqrt{2}}$.
    \item [(c)] \textbf{3-soliton solution}:
    \begin{align}
        u(x, t) = \beta_1\sech^2(\xi_1)- \frac{2(\beta_2-\beta_3)\left(\frac{2(\beta_3-\beta_1)\left(\beta_3 \csch^2(\xi_3)-\beta_1 \sech ^2(\xi_1)\right)}{\left(\sqrt{2 \beta_3} \tanh(\xi_3)-\sqrt{2 \beta_1} \tanh(\xi_1)\right)^2}-\frac{2(\beta_1-\beta_2)\left(\beta_2 \csch^2(\xi_2)+\beta_1 \sech ^2(\xi_1)\right)}{\left(\sqrt{2 \beta_1} \tanh(\xi_1)-\sqrt{2 \beta_2} \coth(\xi_2)\right)^2} \right)}{\left(\frac{2(\beta_1-\beta_2)}{\sqrt{2 \beta_1} \tanh(\xi_1)-\sqrt{2 \beta_2} \coth(\xi_2)}-\frac{2(\beta_3-\beta_1)}{\sqrt{2 \beta_3} \tanh(\xi_3)-\sqrt{2 \beta_1} \coth(\xi_1)}\right)^2}\;,
    \end{align}
    where $\beta_1 = 0.4$, $\beta_2 = 0.7$, $\beta_3 = 1$, $\xi_1 = \frac{\sqrt{\beta_1}(x-2 \beta_1 t)}{\sqrt{2}}$, $\xi_2 = \frac{\sqrt{\beta_2}(x-2 \beta_2 t)}{\sqrt{2}}$, and $\xi_3 = \frac{\sqrt{\beta_3}(x-2 \beta_3 t)}{\sqrt{2}}$.
\end{enumerate}

\section*{Acknowledgments and Funding}
This work was supported by funding from the King Abdullah University of Science and Technology.
\section*{Data Availability}
The datasets and source code generated and analyzed during the current study are available in \cite{MRRK_code}.
\section*{Declarations}
On behalf of all authors, the corresponding author declares that they have no conflict of interest.

\vspace{1.5em}
\bibliographystyle{plain}
\bibliography{references}

\vspace{1.5em}
\end{document}